%% file: paper.tex
\documentclass{article}
\usepackage{graphicx}
\usepackage{latexsym,amsmath,amsfonts,amscd, amsthm}
\usepackage{changebar}
\usepackage{color}
\usepackage{bm}
\usepackage{tikz}
\usepackage{multirow}
\usetikzlibrary{arrows,backgrounds,snakes}

\newtheoremstyle{plainNoItalics}{}{}{\normalfont}{}{\bfseries}{.}{ }{}

\theoremstyle{plain}
\newtheorem{thm}{Theorem}[section]

\theoremstyle{plainNoItalics}

\newtheorem{lem}[thm]{Lemma}
\newtheorem{defn}[thm]{Definition}
\newtheorem{rem}[thm]{Remark}
\newtheorem{prop}[thm]{Proposition}

\newcommand{\beq}{\begin{equation}}
\newcommand{\eeq}{\end{equation}}
\newcommand{\beqa}{\begin{eqnarray}}
\newcommand{\eeqa}{\end{eqnarray}}
\newcommand{\bit}{\begin{itemize}}
\newcommand{\eit}{\end{itemize}}
\newcommand{\bedef}{\begin{defn}}
\newcommand{\edefn}{\end{defn}}
\newcommand{\bpro}{\begin{prop}}
\newcommand{\epro}{\end{prop}}

\newcommand{\df}{\partial}

\newcommand{\Dt}{\Delta t}

\newcommand{\mE}{\mathcal E}

\begin{document}

%\baselineskip=1.6pc

%\vspace*{.10in}

%=============  title  =========================
\input{title}
\input{intro}
\input{macro_micro_v1}

\input{DG_IMEX}

\input{analysis1}

\input{analysis2}

\input{analysis3}

\input{acknowledge}

\bibliographystyle{siam}
\bibliography{refer}

\end{document}

%% file: title.tex
\begin{center}
{\bf
Analysis of Asymptotic Preserving DG-IMEX Schemes
for Linear Kinetic Transport Equations in a Diffusive Scaling
}
\end{center}
\vspace{.2in}

\centerline{
Juhi Jang \footnote{Department of Mathematics, University of California Riverside, Riverside, CA 92521. E-mail: juhijang@math.ucr.edu. Supported in part by NSF grants DMS-0908007 and DMS-1212142.}
Fengyan Li \footnote{Department of Mathematical Sciences, Rensselaer Polytechnic Institute, Troy, NY 12180. E-mail: lif@rpi.edu. Supported in part by NSF CAREER award DMS-0847241 and NSF DMS-1318409.}
Jing-Mei Qiu \footnote{Department of Mathematics, University of Houston,
Houston, 77204. E-mail: jingqiu@math.uh.edu.
The third and fourth authors are supported in part by Air Force Office of Scientific Computing YIP grant FA9550-12-0318, NSF grant DMS-1217008 and University of Houston.}
Tao Xiong \footnote{Department of
Mathematics, University of Houston, Houston, 77204. E-mail:
txiong@math.uh.edu}
}

\bigskip
\centerline{\bf Abstract}

\smallskip

In this paper, some theoretical aspects will be addressed for the asymptotic preserving DG-IMEX schemes recently proposed in \cite{JLQX_numerical} for kinetic transport equations under a diffusive scaling. We will focus on the methods that are based on discontinuous Galerkin (DG) spatial discretizations with the $P^k$ polynomial space and a first order IMEX temporal discretization, and apply them to two linear models:  the telegraph equation and the one-group transport equation in slab geometry. In particular, we will establish uniform numerical stability with respect to  Knudsen number $\varepsilon$ using energy methods, as well as error estimates for any given $\varepsilon$. When $\varepsilon\rightarrow 0$, a rigorous asymptotic analysis of the schemes is also obtained. Though the methods and the analysis are presented for one dimension in space, they can be generalized to higher dimensions directly.

%\vfill

%% file: intro.tex
\section{Introduction}
\label{sec1}
\setcounter{equation}{0}
\setcounter{figure}{0}
\setcounter{table}{0}

Kinetic theory is at the center of multi-scale modeling connecting the invisible microscopic models  with the macroscopic hydrodynamic models. In particular, when the mean free path of particles is sufficiently small, the system is close to the equilibrium state and a macroscopic model is a good approximation to the kinetic equation. Building a passage from kinetic to macroscopic models is a very interesting problem and there has been a lot of mathematical progress over the decades \cite{bardos1991fluid,  saint2009hydrodynamic}.

Designing accurate and efficient numerical schemes for kinetic equations with a broad range of Knudsen number $\varepsilon$ has been an active research area for more than two decades. Among many multi-scale approaches, asymptotic preserving (AP) methods are known to be able to effectively deal with multi-scales and capture hydrodynamic macro-scale limits. Specifically, the schemes are designed to mimic the asymptotic limit from the kinetic to the hydrodynamic models on the PDE level. As a result, the scheme in the limit of $\varepsilon\rightarrow0$
%with under-resolved computational meshes
becomes a consistent discretization of the limiting macro-scale equations.

AP schemes have been intensively studied in different settings and under different scalings, for example for stationary problems \cite{larsen1987asymptotic, morel1989asymptotic, guermond2010asymptotic} and for time dependent problems with hyperbolic and diffusive scalings \cite{jin2010asymptotic, pareschi2011efficient}. Below, to our best knowledge, we briefly review some existing AP methods for kinetic equations in the diffusive limit.  It was firstly shown in \cite{jin1996numerical, naldi1998numerical} that an improper treatment of spatial discretizations even with a stable implicit time discretization may fail to capture the correct asymptotic limit with an under-resolved computational mesh. In \cite{jin1998diffusive, naldi1998numerical, jin2000uniformly}, some proper splitting between the convection and stiff source terms was introduced for AP properties. Later, AP schemes were designed via different multi-scale approaches, e.g. by a standard perturbation procedure  \cite{klar1998asymptotic}, by moment closure approaches \cite{carrillo2008numerical}, by a micro-macro decomposition of kinetic transport equations \cite{lemou2010new} and by projective integration \cite{lafitte2012asymptotic}.  AP schemes have also been designed with different high order discretization strategies such as the discontinuous Galerkin framework \cite{lowrie2002methods}, weighted essentially non-oscillatory (WENO) methods as well as  globally stiffly accurate implicit-explicit (IMEX) schemes \cite{boscarino2011implicit, boscarino2013flux}. Despite the fact that much computational effort was made in designing various AP schemes, rigorous proofs for uniform stability, error estimates and AP properties for fully discrete schemes are relatively rare. For AP schemes based on the micro-macro decomposition, there are some theoretical results.
In \cite{lemou2010new},  a von Neumann analysis was conducted for numerical stability of a first order AP scheme applied to the two-velocity telegraph equation. Stability and error estimate for more general problems were obtained in \cite{liu2010analysis} based on energy methods. In the setting of stationary problems, a rigorous asymptotic analysis was presented for an upwind discontinuous Galerkin discretization, together with the convergent property of the limiting schemes under mesh refinement, when solving the radiative transport equation \cite{guermond2010asymptotic}.

Recently, a family of high order schemes was proposed in \cite{JLQX_numerical} for several linear and nonlinear {\em discrete-velocity} kinetic transport equations under a diffusive scaling.
The schemes are defined for a reformulated set of equations which is obtained from a micro-macro decomposition of the problem just as in \cite{lemou2010new}, with the idea originally proposed in \cite{liu2004boltzmann}. Based on such reformulation,  discontinuous Galerkin (DG) spatial discretizations of arbitrary order of accuracy are applied with suitable numerical fluxes, and in time, we employ globally stiffly accurate high order IMEX Runge-Kutta (RK) methods \cite{boscarino2011implicit} equipped with a carefully chosen implicit-explicit strategy. A formal asymptotic analysis shows that the proposed methods, as Knudsen number $\varepsilon$ goes to $0$, become consistent  high order discretizations for the limiting macro-scale equations. Numerical results presented in  \cite{JLQX_numerical} also demonstrate the stability and high order accuracy of the proposed schemes when $\varepsilon$ is of order $1$ and in the limit of $\varepsilon$ going to $0$.

The current paper follows up the work in \cite{JLQX_numerical} and addresses some theoretical aspects of the proposed methods. In particular, we will establish uniform stability, error estimates, and perform a rigorous asymptotic analysis for the fully discrete scheme when DG spatial discretizations using the $P^k$ polynomial space are coupled with a first order IMEX time discretization.  Two families of linear kinetic transport equations are considered: the two-velocity telegraph equation, and the one-group transport equation in slab geometry with a continuous velocity field, for which the method in \cite{JLQX_numerical} can be directly formulated and applied.
In this work, uniform time step constraint with respect to $\varepsilon$ is established for numerical stability by using energy methods. More specifically, for the DG method with the $P^0$ polynomial space, optimal time step restriction is achieved as in \cite{lemou2010new, liu2010analysis}, namely $\Delta t = O(h^2)$ in the diffusive regime with $\varepsilon\ll1$, and $\Delta t = O(\varepsilon h)$  in the convective regime with $\varepsilon=O(1)$; for the method with the $P^k$ ($k\ge1$) polynomial space, the time step restriction is $\Delta t = O(h^2)$, which is not the most desired condition in the convective regime. It is expected that extending the stability analysis to high order IMEX schemes can relieve the time step constraint for  the method with the $P^k$ ($k\ge1$) space \cite{zhang2010stability}.
 When higher order temporal discretizations are used, most techniques to analyze the spatial discretizations can be carried over, yet new difficulties will arise related to the high order time integrations \cite{zhang2004error, zhang2010stability}. Such analysis is a subject of our future investigation. Based on numerical stability and approximation properties of the discrete spaces, error estimates are further established for the schemes with different choices of numerical fluxes. The results confirm the high order accuracy in space and the first order accuracy in time of the methods. By using the weakly sequential compactness of Hilbert space (which in the present work is either the finite-dimensional discrete space or $L^2(\Omega_v)$ with $\Omega_v$ given in Section \ref{sec2}), we also prove that the proposed schemes will converge to some consistent discretizations of the limiting heat equation when $\varepsilon\rightarrow 0$. Though the methods and  the analysis are presented here and in \cite{JLQX_numerical} for one dimension in space, they can be generalized to higher dimensions directly.

The paper is organized as follows. In Section \ref{sec2}, we introduce the kinetic transport equation in a diffusive scaling, provide its micro-macro decomposition and diffusive limit, and review a class of DG-IMEX schemes introduced in \cite{JLQX_numerical}. In Section \ref{sec:analysis}, stability analysis,  error estimates, and a rigorous asymptotic analysis are carried out for the schemes and discussed in various settings.
%Though the primary focus of this paper is on analysis and only a first order time discretization is considered, the paper itself is a companion of \cite{JLQX_numerical}. For the completeness, in Section \ref{sec:numerical}, we present numerical results by the methods with various spatial and temporal accuracy to solve the one-group transport equation in slab geometry, which have not been presented in \cite{JLQX_numerical}.

%
% 

%% file: macro_micro_v1.tex
%\section{macro-micro}
\section{Formulation}
\label{sec2}
\setcounter{equation}{0}
\setcounter{figure}{0}
\setcounter{table}{0}

We consider the following linear kinetic transport equation in a diffusive scaling
\begin{equation}\label{f-JH}
\varepsilon f_t + v \partial_x f = \frac{1}{\varepsilon}\left(\langle f\rangle - f\right)
\end{equation}
with the initial data $f_0$ and spatially periodic boundary conditions, where $f=f(x,v,t)$ is the distribution function of particles that depends on time $t>0$, position $x\in\Omega_x\subset {\mathbb R}$, and velocity $v\in \Omega_v$.
The parameter $\varepsilon>0$ measures the distance of the system to the equilibrium state and it can be regarded  as  the mean free path of the particles; when $\varepsilon$ is small, the system is close to equilibrium; when $\varepsilon$ is large, the system is far from equilibrium. The operator $\langle f\rangle - f$ is the normalized scattering operator. Here  $\langle f\rangle=\int_{\Omega_v} f d \mu$, and $\mu$ is a measure associated with the velocity space $\Omega_v$ and it will be specified next for each model. In this paper, we discuss two important families of the problem \eqref{f-JH}: the telegraph equation and the one-group transport equation in slab geometry. 

\bigskip
\noindent
$\bullet$ {\underline{ Telegraph equation }} is a discrete-velocity kinetic model  with $\Omega_v=\{-1, 1\}$, namely, $v$ either takes value $-1$ or $1$, and $\mu$ is a discrete
%Lebesgue
measure on $\{-1,1\}$ such that
\[
\langle f\rangle=\int_{\Omega_v} f d \mu: =\frac12(f(x,v=-1,t) + f(x,v=1,t)).
\]

\smallskip
\noindent
$\bullet$ {\underline{ The one-group transport equation in slab geometry }} is a kinetic equation \eqref{f-JH}
where the velocity space, $\Omega_v=[-1,1]$, is continuous. In addition, 
%is taken as  the set of cosine angles of the velocities, 
$d\mu=\frac12 dv$, with $dv$ being the standard Lebeque measure on $\Omega_v$, and 
\[
\langle f\rangle=\int_{\Omega_v} f d \mu: =\frac12\int_{-1}^1 f(v) dv.
\]
 The scattering operator on the right side of \eqref{f-JH}  can be of more general form:
$(Lf)(v)=\int_{-1}^1   s(v,v')(f(v')-f(v)) dv' $ where the kernel $s$ satisfies  $0<s_m\leq s(v,v')\leq s_M$ for all $v,v'\in [-1,1] $, $\int_{-1}^1 s(v,v')dv' =1$, and $s(v,v')=s(v',v)$.
For such kernel $s$, one can deduce that $\langle L\phi \rangle=0$ for all $\phi\in L^2[-1,1]$ and $\langle \phi L\phi \rangle\leq - 2s_m  \langle \phi^2 \rangle $ for $\phi\in \mathcal{N}(L)^\perp$ \cite{bardos1984diffusion}. Of course, the simplest case of such $s$ is $s(v,v')=1/2$, which is the case of \eqref{f-JH}. We refer to \cite{cercignani1988boltzmann} for more detailed discussions on the linear transport equation and \cite{liu2010analysis} for the one-group transport equation in slab geometry.

\bigskip
It turns out that dealing with discrete velocity or continuous velocity does not affect much the formulation of the numerical methods and the theoretical results. We will treat both cases in a uniform setting and point out the differences when necessary. Let us consider the Hilbert space $L^2(\Omega_v; d\mu)$ in $v$ variable with the inner product:
$\langle f , g  \rangle := \int f g d\mu = \langle fg \rangle$, and let  $\Pi$ be  the orthogonal projection operator onto $\text{Span} (1)$,
defined as $\Pi: f \mapsto \Pi f=\langle f \rangle$. Let $ \rho:=  \Pi f=\langle f \rangle$
 denote the macroscopic density for $f$ and we write
\beq
\label{eq:f:mmd:LJ}
f = \langle f \rangle  + \varepsilon g = \rho+\varepsilon g
\eeq
where $\langle g\rangle = 0$. We recall the micro-macro formulation for \eqref{f-JH} in \cite{JLQX_numerical}, motivated by \cite{lemou2010new, liu2004boltzmann},
\begin{equation}
\begin{split}\label{eq:mmd:LJ}
&\partial_t\rho+\partial_x \langle vg \rangle  =0,\\
&\partial_t g  +\frac{1}{\varepsilon}  (\mathbf{I}-\Pi)(v\partial_xg) +\frac{1}{\varepsilon^2 }v\partial_x\rho = -\frac{1}{\varepsilon^2 }g.
\end{split}
\end{equation}
The operator ${\bf I}$ used here is the identity operator.
It is easy to verify that $\partial_t\rho=\partial_x\left(\langle v^2 \rangle\partial_{x}\rho\right)+O(\varepsilon)$ which leads to the linear diffusion equation as $\varepsilon\rightarrow 0$. We note that $\langle v^2 \rangle=1$ for the telegraph equation and $\langle v^2 \rangle=1/3$ for the one-group transport equation in slab geometry.

%% file: DG_IMEX.tex
%\section{DG-IMEX Methods}
%\label{sec:DG-IMEX:3}

%\setcounter{equation}{0}
%\setcounter{figure}{0}
%\setcounter{table}{0}

% -------Macros introduced by Li -------------
\newcommand{\eps}{\varepsilon}
\newcommand{\mD}{{\mathcal D}}
\newcommand{\Ox}{{\Omega_x}}
\newcommand{\Ov}{{\Omega_v}}
\newcommand{\xL}{{x_{i-\frac{1}{2}}}}
\newcommand{\xR}{{x_{i+\frac{1}{2}}}}
\newcommand{\iL}{{i-\frac{1}{2}}}
\newcommand{\iR}{{i+\frac{1}{2}}}
\newcommand{\testR}{{\phi}}   % test function for rho
\newcommand{\testG}{{\psi}}   % test function for g
\newcommand{\mC}{{\mathcal{C}}}
\newcommand{\bI}{{\bf{I}}}

% 2013-7-16

\bigskip
Recently in \cite{JLQX_numerical}, a family of asymptotic preserving methods were proposed for the telegraph equation  based on its micro-macro decomposition \eqref{eq:mmd:LJ}. The methods are of formal high order accuracy in both space and time. They involve discontinuous Galerkin (DG) spatial discretizations and globally stiffly accurate implicit-explicit (IMEX) Runge-Kutta methods in time.   In the limit of $\varepsilon\rightarrow0$, a formal asymptotic analysis \cite{JLQX_numerical} shows that the limiting schemes are
consistent high order discretizations for the limiting linear heat equation.
Though not discussed in   \cite{JLQX_numerical}, both the methods and the formal analysis can be naturally extended to the one-group transport equation in slab geometry where the velocity field is continuous.
Next we will present the formulation of the methods applied to \eqref{f-JH}, before establishing numerical stability, error estimates, and a rigorous asymptotic analysis in the following section.

Let's first introduce some notation. Start with $\{\xR\}_{i=0}^{i=N}$, a partition of $\Omega_x=[x_\textrm{min}, x_\textrm{max}]$. Here $x_{\frac12}=x_\textrm{min}$, $x_{N+\frac12}=x_\textrm{max}$, each element is denoted as $I_i=[\xL, \xR]$ with its length $h_i$, and $h=\max_i h_i$. Given any non-negative integer $k$, we define a finite dimensional discrete space
\begin{equation}
U_h^k=\left\{u\in L^2(\Omega_x): u|_{I_i}\in P^k(I_i), \forall i\right\},
\label{eq:DiscreteSpace:1mesh}
\end{equation}
where the local space $P^k(I)$ consists of polynomials of degree at most $k$ on $I$. Note functions in $U_h^k$ are piecewise-defined, we further denote jump and average of $u$ at $x_\iR$, $\forall i$, as $[u]_\iR={u(x_\iR^+)-u(x_\iR^-)}$ and $\{u\}_\iR=\frac12(u(x_\iR^+)+u(x_\iR^-))$, respectively. Here $u(x^\pm)=\lim_{\Delta x\rightarrow 0^\pm} u(x+\Delta x)$, and we also use $u_\iR=u(x_\iR)$, $u^\pm_\iR=u(x^\pm_\iR), \forall i$.

With the same DG spatial discretization proposed in \cite{JLQX_numerical}, a family of semi-discrete methods are given below for  the micro-macro system \eqref{eq:mmd:LJ}. Look for $\rho_h(\cdot,t), g_h(\cdot, v, t) \in U_h^k$, such that $\forall \testR, \testG \in U_h^k$,
\begin{subequations}
\label{eq:SDG:c}
\begin{align}
&(\df_t \rho_h, \testR)+a_h(g_h, \testR)=0, \label{eq:SDG:1:c}\\
&(\df_t g_h, \testG)+ \frac{1}{\eps} b_{h,v}(g_h, \testG)-\frac{v}{\eps^2} d_h(\rho_h, \testG)
=-\frac{1}{\eps^2}(g_h, \testG),
\label{eq:SDG:2:c}
\end{align}
\end{subequations}
where
\begin{subequations}
\begin{align}
\label{eq:ah}
a_h(g_h,\testR)&=-\sum_i \int_{I_i} \langle vg_h\rangle \df_x\testR dx - \sum_i \widehat{\langle vg_h\rangle}_{\iL} [\testR]_\iL,\\
\label{eq:bh}
b_{h,v}(g_h,\testG)&=((\bI-\Pi)\mD_h(g_h; v), \testG) =(\mD_h(g_h; v) - \langle\mD_h(g_h; v)\rangle, \testG),\\
\label{eq:dh}
d_h(\rho_h, \testG)&=\sum_i\int_{I_i} \rho_h \df_x\testG dx + \sum_i \widehat{\rho}_{h,\iL}[\testG]_\iL \;.
%\; s_h(g_h, \testG)=\int_{\Omega_x} g_h \testG dx.
\end{align}
\end{subequations}
 Here and below, the standard inner product $(\cdot, \cdot)$ for the $L^2(\Omega_x)$ space is used, see e.g. the first term in \eqref{eq:SDG:1:c} and in \eqref{eq:SDG:2:c}. The function $\mD_h(g_h; v)$ in \eqref{eq:bh} belongs to $U_h^k$, and it is defined based on an upwind discretization of $v \df_x g$ within the DG framework,
\begin{align}
(\mD_h(g_h; v), \testG):
=-\sum_i\left(\int_{I_i} vg_h \df_x\testG dx\right) - \sum_i \widetilde{(vg_h)}_\iL[\testG]_\iL, \quad \psi\in U_h^k,
\label{eq:mD}
\end{align}
with $\widetilde{vg}$ being an upwind numerical flux consistent to $vg$,
\beq
\label{eq:vg:upwind:L-1}
\widetilde{vg}:=
\left\{
\begin{array}{ll}
v g^-,&\mbox{if}\; v>0\\
v g^+,&\mbox{if}\; v<0
\end{array}
\right.
=v\{g\}-\frac{|v|}{2}[g].
\eeq

Both $\widehat{\langle vg\rangle}$ and $\hat{\rho}$ in \eqref{eq:ah} and \eqref{eq:dh} are also numerical fluxes, and they are consistent to the physical fluxes $\langle vg\rangle$ and $\rho$.  In this paper, the following choices are considered:
\begin{subequations}
\label{eq:flux}
\begin{align}
\label{eq:flux:1}
{\textrm{alternating:}} & \;\;\;\; \widehat{\langle vg\rangle} = {\langle vg\rangle}^-,\; \hat{\rho} = {\rho}^+\;\textrm{(left-right)}; \;\;\; \textrm{or}\; \widehat{\langle vg\rangle}={\langle vg\rangle}^+, \;\hat{\rho}={\rho}^-\;\textrm{(right-left)},\\
\label{eq:flux:2}
{\textrm{central:}} & \;\;\;\; \widehat{\langle vg\rangle}=\{\langle vg\rangle\}, \; \hat{\rho}=\{\rho\}\;.
\end{align}
\end{subequations}
\begin{rem}
The spatial discretization given above as proposed in \cite{JLQX_numerical} differs from the methods proposed in \cite{lemou2010new} in several aspects.
First of all, the spatial discretization in \cite{lemou2010new} is first order and of finite difference type; while the method considered here is of finite element type, which is known to be more compact to achieve high order accuracy in a systematic manner. Moreover, our methods are based on one set of computational mesh and that in \cite{lemou2010new} is on staggered meshes. In fact, DG spatial discretizations can also be formulated on staggered meshes as in \cite{lemou2010new}. This, on one hand, saves one from using numerical fluxes $\widehat{\langle vg\rangle}$ and $\hat{\rho}$ at grid points (see \eqref{eq:ah}, \eqref{eq:dh}, \eqref{eq:flux}), and on the other hand, the resulting method only has suboptimal $k$-th order accuracy when the $P^k$ polynomial space is used with odd $k$ (this is not reported yet observed numerically, and it is similar to our method with central fluxes for $\widehat{\langle vg\rangle}$ and $\hat{\rho}$), while better spatial accuracy can be achieved both theoretically and numerically by suitably designing $\widehat{\langle vg\rangle}$ and $\hat{\rho}$ within the one-mesh framework.
\end{rem}

The semi-discrete method in \eqref{eq:SDG:c} will be further coupled with globally stiffly accurate IMEX Runge-Kutta methods in time \cite{boscarino2011implicit}. Such temporal discretizations are employed to deal with the stiffness of \eqref{eq:mmd:LJ} when $\varepsilon$ is small, and to ensure the correct asymptotic property of the scheme as $\varepsilon\rightarrow0$. Below we will give the fully discrete scheme with a {\em first} order  globally stiffly accurate IMEX Runge-Kutta method in time, termed as DG-IMEX1, which will be analyzed in the present paper.
Methods with higher order temporal accuracy are formulated for the telegraph equation in \cite{JLQX_numerical}, where the stability and accuracy are demonstrated numerically. Such high order temporal discretizations can also be defined for the one-group transport equation.
 %(also see Section \ref{sec:numerical} for numerical examples).
 The analysis for the fully discrete methods with higher than first order temporal accuracy is expected to be much more involved (see e.g. for such analysis in \cite{zhang2004error,zhang2010stability} when the temporal discretizations are explicit) and it will be investigated in the next stage of our project.

Given $\rho_h^n(\cdot),\; g_h^n(\cdot, v)\in U_h^k$ that approximate the solution $\rho$ and $g$ at $t=t^n$, we look for $\rho_h^{n+1}(\cdot), \;g_h^{n+1}(\cdot,v) \in U_h^k$, such that $\forall\; \phi, \psi\in U_h^k$,
\begin{subequations}
\label{eq:FDG:1T}
\begin{align}
&\left(\frac{\rho_h^{n+1}-\rho_h^n}{\Dt}, \phi\right) + a_h(g_h^n, \phi)=0, \label{eq:FDG:1T:r}\\
&\left(\frac{g_h^{n+1}-g_h^n}{\Dt}, \psi\right)+ \frac{1}{\eps} b_{h,v}(g_h^n, \psi)-\frac{v}{\eps^2} d_h(\rho_h^{n+1}, \psi) =
 -\frac{1}{\eps^2} (g_h^{n+1}, \psi).
% s_{h}(g_h^{n+1}, \psi).
 \label{eq:FDG:1T:g}
\end{align}
\end{subequations}
The most stiff terms, in both the convective and collisional terms, with a scale of $\frac{1}{\varepsilon^2}$ are treated implicitly here.

\begin{rem}
The implicit-explicit strategy used in our temporal discretization is different from that in \cite{lemou2010new}.   We consider it more natural to treat implicitly the most stiff terms with the scale of $\frac{1}{\varepsilon^2}$ in the micro-macro reformulation, and there is no essential change in the computational complexity. Specifically, one can solve  \eqref{eq:FDG:1T:r} first for $\rho_h^{n+1}$, then \eqref{eq:FDG:1T:g} for $g_h^{n+1}$ from a block-diagonal system, indicating our implicit-explicit strategy results in comparable computational complexity as that in \cite{lemou2010new}. Moreover, this implicit-explicit strategy is especially important when it is combined with high order globally stiffly accurate temporal discretizations in \cite{JLQX_numerical}: it not only ensures the limiting schemes as $\varepsilon\rightarrow 0$ to be consistent high order discretizations for the limiting equations, but also preserves the limiting equilibrium on the discrete level, in the sense that
\begin{equation}
vd_h(\rho_h^*, \psi)=(g_h^*, \psi),\quad \psi\in U_h^k,
\end{equation}
is satisfied by $(\rho_h^*, g_h^*)$, the approximating solution from any of the internal stages over one time step or  at any discrete time $t^n$ in the limit of $\varepsilon\rightarrow 0$ (see also Section 3.2 in \cite{JLQX_numerical}).
%3.14) and (3.15) in \cite{JLQX_numerical},  for both internal stages over one time step and at discrete times.
The difference in the implicit-explicit strategy calls for a non-conventional definition of the discrete energy in Theorem \ref{eq:thm1}.  In Remark \ref{rem:3.5} of Section \ref{sec:analysis}, there is also some discussion about how the theoretical results here and in \cite{lemou2010new, liu2010analysis} are related.
\end{rem}

%% file: analysis1.tex
\section{Theoretical Results: Stability, Error Estimates, and Rigorous Asymptotic Analysis}
\label{sec:analysis}

In this section, stability and error estimates will be established for DG-IMEX1 method in \eqref{eq:FDG:1T} for both the telegraph equation and the one-group transport equation in slab geometry based on their micro-macro formulation \eqref{eq:mmd:LJ}.
One will see that uniform stability result is obtained with respect to $\varepsilon$. In terms of accuracy, the method is first order in time for any given $\varepsilon$. The accuracy in space is higher when polynomials of higher degree are used for spatial approximations.   When the exact solutions are bounded uniformly with respect to $\eps$ in certain norms, the error estimates are also uniform in $\eps$. In addition, a rigorous asymptotic analysis is presented for the proposed methods when $\eps\rightarrow 0$.

Without loss of generality, the mesh is assumed to be uniform with $h=h_i, \forall i$. Our results can be extended to general meshes when $\frac{\max_i h_i}{\min_i h_i}$ is uniformly bounded during the mesh refinement. Moreover, with very little change, our analysis can  be established for the one-group transport equation in slab geometry with the more general scattering operator as in Section \ref{sec2} (some constants in the results will also depend on the bounds $s_m$ and $s_M$ of the kernal $s(\cdot,\cdot)$).  For simplicity, we will not present the analysis for the general case. The analysis will be based on the following norms,
\[||\phi||=||\phi||_{L^2(\Ox)}, \qquad |||\phi|||=(\langle ||\phi||^2\rangle)^{1/2}.
 \]
 In the estimates, two standard inverse inequalities will be used \cite{ciarlet:book}. There exist constants $C_{\textrm{inv}}, \hat{C}_{\textrm{inv}}$, such that for any $w\in P^k([a,b])$,
\begin{subequations}
\label{eq:inv}
\begin{align}
&|w(y)|^2 (b-a) \leq C_{\textrm{inv}}\int_a^b w(x)^2 dx, \qquad\mbox{with}\; y=a, \;\mbox{or}\; b \label{eq:inv:1}\\
&(b-a)^2 \int_a^b |w_x(x)|^2 dx \leq \hat{C}_{\textrm{inv}}\int_a^b w(x)^2 dx.
\label{eq:inv:2}
\end{align}
\end{subequations}
The constants $C_{\textrm{inv}}, \hat{C}_{\textrm{inv}}$ are independent of $a$ and $b$, and they depend on $k$ (see \cite{warburton2008taming, reyna2013bound} for the explicit expression of $k$-dependence). Some other basic inequalities, such as Young's inequality
$xy\leq \frac{x^2}{2\zeta}+\frac{\zeta y^2}{2}$ ($\zeta>0$), are used without being pointed out. We also denote
\beq
\label{eq:alpha}
\alpha_1=(||v||_\infty^2+\langle v^2\rangle) \hat{C}_{\textrm{inv}}\;,\quad\alpha_2= 2(||v||_\infty+\langle|v|\rangle)C_{\textrm{inv}}\;,\quad
\alpha_3=2 ||v||_\infty C_{\textrm{inv}}\;,
\eeq
for later use.
%
%%%%%%%%%%%%
% Projection
%%%%%%%%%%%%
%
Below are some orthogonal projections onto $U_h^k$ utilized in our analysis.
\begin{itemize}
\item $L^2$ projection $\pi_h$: $\pi_h w \in U_h^k$, such that
$$\int_{I_i} (\pi_h w - w) v dx=0, \qquad \forall v\in P^k(I_i), \forall i.$$
\item Gauss-Radau projection $\pi_h^-$: $\pi_h^- w \in U_h^k$, such that
$$\int_{I_i} (\pi_h^- w - w) v dx=0, \qquad \forall v\in P^{k-1}(I_i), \forall i$$
 and $(\pi_h^- w)_\iR^-=w(x_\iR^-), \forall i.$
\item Gauss-Radau projection $\pi_h^+$: $\pi_h^+ w \in U_h^k$, such that
$$\int_{I_i} (\pi_h^+ w - w) v dx=0, \qquad \forall v\in P^{k-1}(I_i), \forall i$$
 and $(\pi_h^+ w)_\iL^+=w(x_\iL^+), \forall i.$
\end{itemize}
These projections are commonly used in theoretical analysis of DG methods, and they have the following properties which can be easily established \cite{ciarlet:book},
\beq
||w-\Pi_h w||^2+h \sum_i ((w-\Pi_h w)_\iL^\pm)^2 \leq C h^{2k+2} ||w||_{H^{k+1}(\Omega_x)}^2.
\label{eq:pi:1}
\eeq
Here  $\Pi_h$ is any of $\pi_h, \pi_h^-, \pi_h^+$, and the constant $C$ depends only on $k$.

At $t=0$, we initialize the methods through $L^2$ projection in space, that is, we take $\rho_h^0=\pi_h \rho|_{t=0}$ and $g_h^0=\pi_h g|_{t=0}$. It can be seen easily that
\begin{equation}\label{eq:a:gh0}
\langle g_h^0\rangle=\pi_h\langle g\rangle|_{t=0}=0.
 \end{equation}
 Other types of initialization, for example $\rho_h^0=\Pi_\rho \rho|_{t=0}$ and $g_h^0=\Pi_g g|_{t=0}$ with $\Pi_\rho$ and $\Pi_g$ specified in Section \ref{sec:3.2},  can be considered without any essential change to the results presented in this section.

\subsection{Stability analysis}
\label{sec:3.1}

%@@@@@@@@@@
% Lemma: stability (1)
%@@@@@@@@@@

To prove stability, we first present two lemmas. In particular, Lemma \ref{lem:0} is a discrete analogue of the property $\langle g\rangle=0$ for the exact solution. It plays an important role in both stability analysis and error estimates.
\begin{lem}
\label{lem:0} The numerical solution $g_h^n$ has the following property
\beq
\langle g_h^{n} \rangle=0, \quad \forall n.
\label{eq:lem0:1}
\eeq
\end{lem}
\begin{proof}
Take $\forall \psi\in U_h^k$ in \eqref{eq:FDG:1T:g}, integrate over $v$, with $\langle v \rangle=0$, one has
\beq
\langle \left(\frac{g_h^{n+1}-g_h^n}{\Dt}, \psi\right) \rangle =
 -\frac{1}{\eps^2} \langle (g_h^{n+1}, \psi)\rangle.\notag
\eeq
This can be further organized into
\beq
(\langle g_h^{n+1} \rangle, \psi)=\frac{\eps^2}{\eps^2+\Dt} (\langle g_h^n \rangle, \psi), \forall \psi\in U_h^k.
\label{eq:lem0:3}
\eeq
By taking $\psi=\langle g_h^{n+1} \rangle - \frac{\eps^2}{\eps^2+\Dt} \langle g_h^n \rangle \in U_h^k$, in addition to the fact $\langle g_h^0\rangle=0$ in \eqref{eq:a:gh0}, one can conclude
$\langle g_h^{n+1} \rangle = \frac{\eps^2}{\eps^2+\Dt} \langle g_h^n \rangle$ hence \eqref{eq:lem0:1}.
\end{proof}

%@@@@@@@@@@
% Lemma: stability (2)
%@@@@@@@@@@

\bigskip
\begin{lem}
\label{lem:1}
Let $n$ and $m$ be any non-negative integer indices. For any $\phi^n(\cdot)$, $\phi^{n+1}(\cdot)$, $\psi^m(\cdot,v)$, $\psi^{m+1}(\cdot,v) \in U_h^k$ satisfying $\langle \psi^m \rangle=\langle \psi^{m+1} \rangle=0$, denote
\begin{align}
\Xi_\eps(\phi^{n+1},  \phi^n, \psi^{m+1}, \psi^m):&=
\left(\frac{\phi^{n+1}-\phi^n}{\Dt}, \phi^{n+1}\right) + a_h(\psi^{m+1}, \phi^{n+1})\\
&+\langle \left(\frac{\psi^{m+1}-\psi^m}{\Dt}, \eps^2\psi^{m+1}\right)+\frac{1}{\eps} b_{h,v}(\psi^m, \eps^2\psi^{m+1})-\frac{v}{\eps^2} d_h(\phi^n, \eps^2\psi^{m+1})\rangle,\notag
\end{align}
the following estimates then hold with any of the numerical flux in  \eqref{eq:flux},
\begin{align}
&\Xi_\eps(\phi^{n+1}, \phi^n, \psi^{m+1}, \psi^m) \geq \frac{1}{2\Dt}\left((||\phi^{n+1}||^2+\eps^2|||\psi^{m+1}|||^2)-(||\phi^n||^2+\eps^2|||\psi^m|||^2)\right)  \label{lem:1-0}\\
&+ \left\{
 \begin{array}{ll}
 (\eps-\alpha_2 \frac{\Dt}{h})\langle\frac{|v|}{2}\sum_i [\psi^{m+1}]_\iL^2\rangle  -\alpha_1 \frac{\Dt}{h^2} |||\psi^{m+1}|||^2, &\textrm{for}\; k\geq 1\\
 (\eps-\frac{\alpha_2}{2}\frac{\Dt}{h})\langle\frac{|v|}{2}\sum_i [\psi^{m+1}]_\iL^2\rangle, &\textrm{for}\; k=0,
\end{array}
\right.\notag
\end{align}
with $\alpha_i, i=1, 2, 3$ defined in \eqref{eq:alpha}.
\end{lem}

%%%%%%%%
\begin{proof}
Note that
\begin{align}
 \label{lem:1-1}
&\Xi_\eps(\phi^{n+1}, \phi^n, \psi^{m+1}, \psi^m)\\
&=\frac{1}{2\Dt}\left(||\phi^{n+1}||^2-||\phi^n||^2+||\phi^{n+1}-\phi^n||^2 \right) + \frac{\eps^2}{2\Dt} (|||\psi^{m+1}|||^2- |||\psi^m|||^2+|||\psi^{m+1}-\psi^m|||^2) \notag\\
&+ a_h(\psi^{m+1}, \phi^{n+1})+\eps \langle b_{h,v}(\psi^m, \psi^{m+1})\rangle -\langle v d_h(\phi^{n+1}, \psi^{m+1})\rangle
 + \langle v d_h(\phi^{n+1}-\phi^n, \psi^{m+1}) \rangle.\notag
\end{align}
Based on the definitions of the bilinear forms $a_h(\cdot, \cdot)$ %b_{h,v}(\cdot, \cdot)$,
 and $d_h(\cdot, \cdot)$, one has
\begin{align}
a_h&(\psi^{m+1}, \phi^{n+1})- \langle v d_h(\phi^{n+1}, \psi^{m+1})\rangle \notag\\
=&-\sum_i\int_{I_i} \df_x \left(\langle v\psi^{m+1}\rangle \phi^{n+1}\right) dx-\sum_i \widehat{\langle v\psi^{m+1}\rangle}_{\iL} [\phi^{n+1}]_{\iL}
 - \sum_i\widehat{\phi^{n+1}_\iL}[\langle v \psi^{m+1} \rangle]_\iL\notag\\
=& \sum_i\left([\langle v\psi^{m+1}\rangle \phi^{n+1}]-\widehat{\langle v\psi^{m+1}\rangle}[\phi^{n+1}]
       -\widehat{\phi^{n+1}}[\langle v \psi^{m+1} \rangle]\right)_\iL=0.
\label{lem:1-2}
\end{align}
The last equality can be verified directly with the definition of central and alternating fluxes. In addition,  with $\langle\psi^{m+1}\rangle=0$ and equation \eqref{eq:mD} and the upwind flux \eqref{eq:vg:upwind:L-1},
\begin{align}
&\langle b_{h,v}(\psi^m, \psi^{m+1})\rangle= \langle(\mD_h(\psi^m; v) - \langle\mD_h(\psi^m; v)\rangle, \psi^{m+1})\rangle\notag\\
&= \langle(\mD_h(\psi^m; v), \psi^{m+1})\rangle-(\langle \mD_h(\psi^m; v)\rangle, \langle\psi^{m+1}\rangle)= \langle(\mD_h(\psi^m; v), \psi^{m+1})\rangle \notag\\
&= \langle(\mD_h(\psi^{m+1}; v), \psi^{m+1})\rangle+\langle(\mD_h(\psi^m-\psi^{m+1}; v), \psi^{m+1})\rangle\notag\\
&=-\sum_i\left(\langle \int_{I_i} v\psi^{m+1}\df_x\psi^{m+1} dx + \widetilde{(v\psi^{m+1})}_\iL[\psi^{m+1}]_\iL\rangle\right)+\langle(\mD_h(\psi^m-\psi^{m+1}; v), \psi^{m+1})\rangle\notag\\
&=\langle \frac{|v|}{2}\sum_i [\psi^{m+1}]^2_\iL\rangle+\langle(\mD_h(\psi^m-\psi^{m+1}; v), \psi^{m+1})\rangle .
\label{lem:1-3}
\end{align}
Next we want to estimate $\langle(\mD_h(\psi^m-\psi^{m+1}; v), \psi^{m+1})\rangle=\Lambda_1+\Lambda_2$  in \eqref{lem:1-3}, where
\begin{equation*}
\Lambda_1= \langle \sum_i\int_{I_i} v(\psi^{m+1}-\psi^m)\df_x\psi^{m+1} dx \rangle, \qquad\Lambda_2=\langle \sum_i \widetilde{(v(\psi^{m+1}-\psi^m))}_\iL[\psi^{m+1}]_\iL\rangle.
\end{equation*}
These two terms can be bounded as follows.
\beq
|\Lambda_1|\leq \theta_1|||\psi^{m+1}-\psi^m|||^2+\frac{1}{4\theta_1} \langle \int_\Ox (v \df_x \psi^{m+1})^2 dx\rangle,
\label{lem:1-4}
\eeq
\begin{align}
|\Lambda_2|
=&|\langle \sum_i v(\psi^{m+1}-\psi^m)_\iL^{-}[\psi^{m+1}]_\iL\rangle_{+} +\langle \sum_i v(\psi^{m+1}-\psi^m)_\iL^+[\psi^{m+1}]_\iL\rangle_{-}|\notag\\
\leq& \frac{\eta_1}{C_{\textrm{inv}}} \langle \sum_i h ((\psi^{m+1}-\psi^m)_\iL^{-})^2\rangle_{+}
       +\frac{C_{\textrm{inv}}}{4\eta_1} \langle \sum_i h^{-1} (v[\psi^{m+1}]_\iL)^2\rangle_{+}\notag\\
    & + \frac{\eta_1}{C_{\textrm{inv}}}\langle \sum_i h ((\psi^{m+1}-\psi^m)_\iL^{+})^2\rangle_{-}
       +\frac{C_{\textrm{inv}}}{4\eta_1} \langle \sum_i h^{-1}(v[\psi^{m+1}]_\iL)^2\rangle_{-}\notag\\
\leq& \eta_1|||\psi^{m+1}-\psi^m  |||^2+\frac{ C_{\textrm{inv}}}{4\eta_1} \sum_i h^{-1} \langle (v[\psi^{m+1}]_\iL)^2\rangle.
\label{lem:1-5}
\end{align}
Here the inverse inequality in \eqref{eq:inv:1} is used, and $\langle \cdot \rangle_{+}$ and $\langle \cdot \rangle_{-}$ are integrals with respect to the positive and negative part of $v$, respectively. What we also need to estimate is $\langle v d_h(\phi^{n+1}-\phi^n, \psi^{m+1}) \rangle$.
\begin{align}
\label{lem:1-6}
&|\langle v d_h(\phi^{n+1}-\phi^n, \psi^{m+1}) \rangle|\notag\\
&\leq |\sum_i \int_{I_i}(\phi^{n+1}-\phi^{n}) \df_x\langle v \psi^{m+1}\rangle dx| + |\sum_i \left(\widehat{(\phi^{n+1}-\phi^n)}\langle v[\psi^{m+1}]\rangle \right)_\iL|\notag\\
&\leq  \theta_2 ||\phi^{n+1}-\phi^{n}||^2+\frac{1}{4\theta_2} \int_\Ox (\df_x \langle v \psi^{m+1}\rangle)^2 dx
+\frac{\eta_2}{C_{\textrm{inv}}} \sum_i h (\widehat{\phi^{n+1}-\phi^n})^2_\iL
+\frac{C_{\textrm{inv}}}{4\eta_2}\sum_i h^{-1} \langle v [\psi^{m+1}]_\iL\rangle^2 \notag\\
&\leq  \theta_2 ||\phi^{n+1}-\phi^{n}||^2+\frac{1}{4\theta_2} \int_\Ox (\df_x \langle v \psi^{m+1}\rangle)^2 dx
+\eta_2||\phi^{n+1}-\phi^{n}||^2
+\frac{C_{\textrm{inv}}}{4\eta_2}\sum_i h^{-1} \langle v [\psi^{m+1}]_\iL\rangle^2 \notag\\
&= (\theta_2+\eta_2) ||\phi^{n+1}-\phi^{n}||^2+\frac{1}{4\theta_2} \int_\Ox (\df_x \langle v \psi^{m+1}\rangle)^2 dx
+\frac{C_{\textrm{inv}}}{4\eta_2}\sum_i h^{-1} \langle v [\psi^{m+1}]_\iL\rangle^2 .
\end{align}
Up to now, $\theta_i, \eta_i$ with $i=1, 2$ are arbitrary positive constants.
By further applying inverse inequalities in \eqref{eq:inv:2} and the following simple estimates,
\begin{align*}
\langle \int_\Ox(v \df_x \psi^{m+1})^2 dx\rangle
&\leq ||v||_\infty^2 \langle \int_\Ox(\df_x \psi^{m+1})^2 dx\rangle
 \leq \frac{\hat{C}_{\textrm{inv}}}{h^2} ||v||_\infty^2 |||\psi^{m+1}|||^2,\\
\int_\Ox (\df_x \langle v \psi^{m+1}\rangle)^2 dx
&=\int_\Ox \langle v\df_x \psi^{m+1}\rangle^2 dx
  \leq  \int_\Ox \langle v^2\rangle \langle (\df_x\psi^{m+1})^2\rangle dx\\
&=\langle v^2\rangle \langle\int_\Ox (\df_x\psi^{m+1})^2 dx\rangle
  \leq \frac{\hat{C}_{\textrm{inv}}}{h^2} \langle v^2\rangle |||\psi^{m+1}|||^2,\\
\sum_i\langle (v[\psi^{m+1}]_\iL)^2\rangle
&\leq 2||v||_\infty\langle\frac{|v|}{2}\sum_i [\psi^{m+1}]_\iL^2\rangle ,\\
\sum_i \langle v [\psi^{m+1}]_\iL\rangle^2
&\leq\sum_i \langle |v|\rangle \langle |v|([\psi^{m+1}]_\iL)^2\rangle
= 2\langle|v|\rangle \langle \frac{|v|}{2}\sum_i [\psi^{m+1}]_\iL^2\rangle ,
\end{align*}
we have
\begin{align}
|\langle(\mD_h(\psi^m-\psi^{m+1}; v), \psi^{m+1})\rangle|
\leq & (\theta_1+\eta_1)|||\psi^{m+1}-\psi^m|||^2
 +\frac{\hat{C}_{\textrm{inv}}||v||_\infty^2}{4\theta_1 h^2} |||\psi^{m+1}|||^2\notag\\
 &+\frac{C_{\textrm{inv}}||v||_\infty}{2\eta_1 h} \langle\frac{|v|}{2}\sum_i [\psi^{m+1}]_\iL^2\rangle,\label{lem:1-7}
\end{align}
\begin{align}
|\langle v d_h(\phi^{n+1}-\phi^n, \psi^{m+1}) \rangle|
\leq &(\theta_2+\eta_2) ||\phi^{n+1}-\phi^{n}||^2+\frac{\hat{C}_{\textrm{inv}}\langle v^2\rangle}{4\theta_2h^2} |||\psi^{m+1}|||^2\notag\\
&+\frac{C_{\textrm{inv}}\langle|v|\rangle}{2\eta_2 h} \langle \frac{|v|}{2}\sum_i [\psi^{m+1}]_\iL^2\rangle .\label{lem:1-8}
\end{align}

Estimates in \eqref{lem:1-1}-\eqref{lem:1-3} and  \eqref{lem:1-7}-\eqref{lem:1-8} are now assembled together,
\begin{align}
\label{lem:1-9}
&\Xi_\eps(\phi^{n+1}, \phi^n, \psi^{m+1}, \psi^m)\geq \frac{1}{2\Dt}\left((||\phi^{n+1}||^2+\eps^2|||\psi^{m+1}|||^2)-(||\phi^n||^2+\eps^2|||\psi^m|||^2)\right)\notag\\
&+\left(\frac{1}{2\Dt}-(\theta_2+\eta_2)\right)||\phi^{n+1}-\phi^n||^2 + \left(\frac{\eps^2}{2\Dt}-\eps(\theta_1+\eta_1)\right)|||\psi^{m+1}-\psi^m|||^2 \notag\\
 &+\left(\eps-\frac{\eps C_{\textrm{inv}}||v||_\infty}{2\eta_1 h}-\frac{C_{\textrm{inv}}\langle|v|\rangle}{2\eta_2 h}\right) \langle\frac{|v|}{2}\sum_i [\psi^{m+1}]_\iL^2\rangle\notag\\
 &-\left(\frac{\eps\hat{C}_{\textrm{inv}}||v||_\infty^2}{4\theta_1 h^2}+\frac{\hat{C}_{\textrm{inv}}\langle v^2\rangle}{4\theta_2h^2}\right) |||\psi^{m+1}|||^2.
 \end{align}
One can conclude the lower bound of $\Xi_\eps$  in \eqref{lem:1-0}
%, and therefore \eqref{lem:1-0-0} for the discrete-velocity case,
by taking $\theta_1=\eta_1=\frac{\eps}{4\Dt}$, $\theta_2=\eta_2=\frac{1}{4\Dt}$ in \eqref{lem:1-9}. When $k=0$, $\df_x\psi^{m+1}=0$ and all terms involving $\theta_1$ and $\theta_2$ are no longer needed in the analysis. The estimate in \eqref{lem:1-9} turns to
\begin{align}
\label{lem:1-10}
&\Xi_\eps(\phi^{n+1}, \phi^n, \psi^{m+1}, \psi^m)\geq \frac{1}{2\Dt}\left((||\phi^{n+1}||^2+\eps^2|||\psi^{m+1}|||^2)-(||\phi^n||^2+\eps^2|||\psi^m|||^2)\right)\notag\\
&+\left(\frac{1}{2\Dt}-\eta_2\right)||\phi^{n+1}-\phi^n||^2 + \left(\frac{\eps^2}{2\Dt}-\eps\eta_1\right)|||\psi^{m+1}-\psi^m|||^2 \notag\\
 &+\left(\eps-\frac{\eps C_{\textrm{inv}}||v||_\infty}{2\eta_1 h}-\frac{C_{\textrm{inv}}\langle|v|\rangle}{2\eta_2 h}\right) \langle\frac{|v|}{2}\sum_i [\psi^{m+1}]_\iL^2\rangle.
 \end{align}
In particular, with $\eta_1=\frac{\eps}{2\Dt}$ and $\eta_2=\frac{1}{2\Dt}$ in \eqref{lem:1-10}, we conclude the estimate for $k=0$.
%\eqref{lem:1-0:k0}.
\end{proof}

%@@@@@@@@@@
% Thm: stability
%@@@@@@@@@@

\begin{thm}[Stability of DG-IMEX1] When the DG-IMEX1 method \eqref{eq:FDG:1T} is applied to the kinetic transport equation \eqref{f-JH} in its micro-macro decomposition formulation \eqref{eq:mmd:LJ},
the following stability result holds for the numerical solution,
\beq
\label{eq:thm1:0}
||\rho_h^{n+1}||^2+\eps^2|||g_h^n|||^2 \leq ||\rho_h^{n}||^2+\eps^2|||g_h^{n-1}|||^2, \;\forall n
\eeq
under the condition
 \beq
 \label{eq:cfl}
 \Dt\leq \Dt_{stab}=\left\{
  \begin{array}{ll}
  \frac{h}{\alpha_1+\alpha_2\alpha_3} (h+\min(\eps, \frac{\alpha_2h}{\alpha_1})\alpha_3), &\mbox{for}\; k\geq 1,\\
  &\\
   \frac{2h}{\alpha_2 \alpha_3}(h+\alpha_3\eps), &\mbox{for}\; k=0.
  \end{array}
 \right.
 \eeq
Here $\alpha_i, i=1, 2, 3$ are defined in \eqref{eq:alpha}.
\label{eq:thm1}
\end{thm}
\begin{proof}
Take $\phi=\rho_h^{n+1}$ in \eqref{eq:FDG:1T:r}. Additionally  take $\psi=\eps^2 g_h^{n+1}$ in \eqref{eq:FDG:1T:g}, integrate the equation over $v$, and shift the index $n$ to $n-1$. This leads to
\begin{subequations}
\label{eq:thm1:1}
\begin{align}
&\left(\frac{\rho_h^{n+1}-\rho_h^n}{\Dt}, \rho_h^{n+1}\right) + a_h(g_h^n, \rho_h^{n+1})=0, \label{eq:thm1:1:r}\\
&\langle \left(\frac{g_h^n-g_h^{n-1}}{\Dt}, \eps^2 g_h^n\right)+ \frac{1}{\eps} b_{h,v}(g_h^{n-1}, \eps^2 g_h^n)-\frac{v}{\eps^2} d_h(\rho_h^n, \eps^2 g_h^n) \rangle =
 -\langle \frac{1}{\eps^2}(g_h^n, \eps^2 g_h^n)\rangle.
 \label{eq:thm1:1:g}
\end{align}
\end{subequations}
Now we sum up equations \eqref{eq:thm1:1:r}-\eqref{eq:thm1:1:g}, denote the left side of the resulting equation as $LHS$, and get
\beq
\label{eq:thm1:2}
LHS=-|||g_h^n|||^2.
 \eeq
Let's first consider $k\geq 1$. For any integer index $n\geq 1$, by applying Lemma \ref{lem:1} with $\phi=\rho_h$, $\psi=g_h$, and  $m=n-1$, one gets
\begin{align}
LHS =\; \Xi_\eps(\rho_h^{n+1}, \rho_h^{n}, g_h^n, g_h^{n-1})
\geq& \frac{1}{2\Dt}\left((||\rho_h^{n+1}||^2+\eps^2|||g_h^n|||^2)-(||\rho_h^{n}||^2+\eps^2|||g_h^{n-1}|||^2)\right) \notag \\
  &+ \left(\eps-\alpha_2\frac{\Dt}{h}\right)\langle\frac{|v|}{2}\sum_i [g_h^n]_\iL^2\rangle  -\alpha_1 \frac{\Dt}{h^2} |||g_h^n|||^2.  \label{eq:thm1:3}
\end{align}
Note that
\beq
\langle\frac{|v|}{2}\sum_i [g_h^n]_\iL^2\rangle
\leq ||v||_\infty\langle \sum_i (g_{h,\iL}^{n,+})^2+(g_{h,\iL}^{n,-})^2\rangle
\leq  \frac{2 ||v||_\infty C_{\textrm{inv}}}{h} |||g_h^{n}|||^2 = \frac{\alpha_3}{h} |||g_h^{n}|||^2.\label{eq:thm1:4}
\eeq
Combining this estimate with \eqref{eq:thm1:2}-\eqref{eq:thm1:3}, one has
\beq
\label{eq:thm1:5}
\frac{1}{2\Dt}\left((||\rho_h^{n+1}||^2+\eps^2|||g_h^n|||^2)-(||\rho_h^{n}||^2+\eps^2|||g_h^{n-1}|||^2)\right)\leq \gamma |||g_h^n|||^2,
\eeq
where
\begin{equation}
\label{eq:gamma}
\gamma=\left(-1+ \alpha_1 \frac{\Dt}{h^2}\right)
+\max\left(-\eps+\alpha_2 \frac{\Dt}{h}, 0\right)\frac{\alpha_3}{h}\;.
\end{equation}
  The stability result \eqref{eq:thm1:0} can be obtained now as long as $\gamma\leq 0$. Equivalently, this requires
\beq
-\eps+\alpha_2 \frac{\Dt}{h} \leq 0 \;\;\mbox{and}\; -1+ \alpha_1 \frac{\Dt}{h^2} \leq 0,
\eeq
or
\beq
-\eps+\alpha_2 \frac{\Dt}{h} > 0 \;\;\mbox{and}\; -1+ \alpha_1 \frac{\Dt}{h^2} + (-\eps+\alpha_2 \frac{\Dt}{h})\frac{\alpha_3}{h} \leq 0\;.
\eeq
 These conditions can be reformulated into

 \bigskip
 {\em Case 1:} when $\frac{\eps h}{\alpha_2} \leq \frac{h^2}{\alpha_1}$, that is $\eps\leq \frac{\alpha_2}{\alpha_1} h$, then
 $\Dt \leq \frac{h^2+\alpha_3 \eps h }{\alpha_1+\alpha_2\alpha_3}$.

 {\em Case 2:} when $\eps > \frac{\alpha_2}{\alpha_1} h$, then $\Dt \leq \frac{h^2}{\alpha_1}$.

\bigskip
 Conditions in both cases  can be compactly written into
 \beq
 \label{eq:Dt}
     \Dt \leq \frac{h}{\alpha_1+\alpha_2\alpha_3} (h+\min(\eps, \frac{\alpha_2h}{\alpha_1})\alpha_3).
 \eeq

 Finally, we consider $k=0$. By following similar analysis as above using the corresponding result in
 Lemma \ref{lem:1}, one will have \eqref{eq:thm1:5} where
\beq
\gamma=-1 +\max\left(-\eps+ \frac{\alpha_2}{2}\frac{\Dt}{h}, 0\right)\frac{\alpha_3}{h}.
\eeq
To conclude the stability result \eqref{eq:thm1:0}, it is required to have $\gamma\leq 0$ which is equivalent to $-1 +\left(-\eps+ \frac{\alpha_2}{2}\frac{\Dt}{h}\right)\frac{\alpha_3}{h}\leq 0$. This condition can be further simplified into $\Dt \leq \frac{2h}{\alpha_2 \alpha_3}(h+\alpha_3\eps)$. This completes the proof.
\end{proof}

Note that the stability is measured in a non-conventional way in that the discrete energy in \eqref{eq:thm1:0} at the $n$-th step consists of the $L^2$ norm of $\rho_h^{n}$, and the $L^2$ norm of $g_h^{n-1}$. This is due to the  implicit-explicit strategy used in the numerical formulation. One can refer to  \cite{lemou2010new} \cite{liu2010analysis} for a different implicit-explicit strategy  used in a first order scheme, hence the stability analysis with different discrete energy. By working out the expression of $\alpha_i, i=1, 2, 3$ in \eqref{eq:alpha} for the specific equations in Section \ref{sec2}, we further have the following remark.
%Corollary.
%
\begin{rem}
The condition \eqref{eq:cfl} is
 \beq
 \label{eq:cfl:1}
 \Dt\leq \Dt_{stab}=\left\{
  \begin{array}{ll}
  \frac{1}{2\hat{C}_\textrm{inv}+8{C}^2_\textrm{inv}} (h+2{C}_\textrm{inv} \min(\eps, \frac{2C_\textrm{inv}}{\hat{C}_\textrm{inv}}h)) h,&\mbox{for}\; k\geq 1,\\
    \frac{h}{4C^2_\textrm{inv}} (h+2C_\textrm{inv}\eps), &\mbox{for}\; k=0.
  \end{array}
 \right.
 \eeq
for the telegraph equation, and it is
 \beq
 \label{eq:cfl:2}
 \Dt\leq \Dt_{stab}=\left\{
  \begin{array}{ll}
  \frac{1}{\frac{4}{3}\hat{C}_\textrm{inv}+6{C}^2_\textrm{inv}} (h+2 C_\textrm{inv} \min(\eps, \frac{9}{4}\frac{C_\textrm{inv}}{\hat{C}_\textrm{inv}}h)) h,&\mbox{for}\; k\geq 1,\\
  \frac{h}{3C^2_\textrm{inv}} (h+2C_\textrm{inv}\eps), &\mbox{for}\; k=0.
  \end{array}
 \right.
 \eeq
for the one-group transport equation in slab geometry.
\end{rem}

\begin{rem}
\label{rem:3.5}
\begin{itemize}
\item The stability condition in \eqref{eq:cfl} for the DG-IMEX1  is  established {\em uniformly} with respect to $\eps$ for any given integer $k\geq 0$.
\item  When  $k=0$, one has ${C}_\textrm{inv}=1$. (The actual values of ${C}_\textrm{inv}$ and $\hat{C}_\textrm{inv}$ for $k\geq 1$ can be found in  \cite{warburton2008taming, reyna2013bound}.)
  The stability condition in \eqref{eq:cfl:1} becomes $\Dt\leq  \frac{1}{4} h^2+ \frac{1}{2} \eps h$, and the one in \eqref{eq:cfl:2} is
     $\Dt\leq  \frac{1}{3} h^2+ \frac{2}{3} \eps h$. These results are the same as that of the first order finite difference method introduced in \cite{lemou2010new} on staggered grids which employs a different implicit-explicit strategy in discretization (see \cite{lemou2010new} for the telegraph equation and \cite{liu2010analysis} for the one-group transport equation in slab geometry). Moreover,  when $\eps$ is small with the equation in the diffusive regime, $\Dt=O(h^2)$; when $\eps$ is large with the equation in the convective regime, $\Dt=O(\eps h)$. The timestep restrictions in both regimes are standard for explicit schemes.

\item For  $k\geq 1$, the restriction on the time step is $\Dt=O(h^2)$, which is reasonable for $\eps\ll1$,  yet not the most desired condition for $\eps=O(1)$ in the convective regime. Similar as in \cite{zhang2010stability} for analyzing DG methods with explicit Runge-Kutta time discretizations, we conjecture that higher order time discretizations will improve or overcome the restrictive condition on the time step when $\eps=O(1)$. This will be left to our future investigation.
\item  For the telegraph equation, with its special discrete velocity space  $v\in \{-1, 1\}$,  one can verify that $(\mathbf{I}-\Pi)(v\partial_xg)=v\langle \partial_x g\rangle=v\partial_x \langle g\rangle$ holds. Since the exact solution satisfies $\langle g\rangle=0$, it seems one does not need to include $(\mathbf{I}-\Pi)(v\partial_xg)$ in the micro-macro formulation \eqref{eq:mmd:LJ}. Numerically this means not to include $b_{h,v}$ term in \eqref{eq:SDG:c} and in \eqref{eq:FDG:1T}. For the resulting scheme, we can follow the similar analysis as in this subsection and obtain the stability result \eqref{eq:thm1:0} under the condition
  \beq
 \label{eq:cfl:1:00}
 \Dt\leq \Dt_{stab}=\left\{
  \begin{array}{ll}
  \frac{1}{\hat{C}_\textrm{inv}+4 C^2_\textrm{inv}} h^2, &\mbox{for}\; k\geq 1,\\
    \frac{1}{2C^2_\textrm{inv}}h^2 =\frac{1}{2} h^2, &\mbox{for}\; k=0.
  \end{array}
 \right.
 \eeq
The time step constraint for stability is no longer reflecting the ``multi-scale'' aspect of the equation, and it is independent of $\eps$ and always $\Dt=O(h^2)$, even in the convective regime  with $\eps=O(1)$ and when $k=0$. This shows the importance of including the term $(\mathbf{I}-\Pi)(v\partial_xg)$ in the design of numerical methods.
\end{itemize}
\end{rem}

%% file: analysis2.tex
\subsection{Error estimates}
\label{sec:3.2}

In this subsection, error estimates are carried out for the proposed method \eqref{eq:FDG:1T} to solve the kinetic transport equation \eqref{f-JH}  in its micro-macro formulation \eqref{eq:mmd:LJ} with smooth exact solutions at any given time $T: 0<T<\infty$. Let $\rho^n$ and $g^n$ be the exact solution at time $t^n=n\Dt$, $\forall n$.  Let $\Pi_\rho$ and $\Pi_g$ denote two orthogonal projections onto $U_h^k$ which will be specified later. Define the error function in $\rho$, $e_\rho^n=\rho^n-\rho_h^n=\xi_\rho^n-\eta_\rho^n$ where $\xi_\rho^n=\Pi_\rho\rho^n-\rho_h^n$ and $\eta_\rho^n=\Pi_\rho\rho^n-\rho^n$.  Similarly, $e_g^n=g^n-g_h^n=\xi_g^n-\eta_g^n$ where $\xi_g^n=\Pi_gg^n-g_h^n$ and $\eta_g^n=\Pi_gg^n-g^n$. We also denote
$\mE_n=||\xi_\rho^n||^2+\eps^2|||\xi_g^{n-1}|||^2$. Both $\eta_\rho^n$ and $\eta_g^n$ can be estimated in a standard way based on the definitions of $\Pi_\rho$, $\Pi_g$, and $U_h^k$ (see the beginning of Section \ref{sec:analysis} regarding the property of projections), therefore the error estimates for the proposed methods boil down to the estimation of $\xi_\rho^n$ and $\xi_g^n$, $\forall n$.   Throughout this subsection, we use $C, C_*>0$ to denote generic constants. Here $C$ only depends on $k$; $C_*$ is independent of $h$, $\Dt$ and $n$ and depends on $k$, $T$, $||v||_\infty$ and some Sobolev norms of the exact solutions, more specifically, an upper bound of
$$||\partial_{tt} \rho||, ||\partial_{xt} \rho||, ||\rho||_{H^{k+1}(\Omega_x)},
||\partial_t\rho||_{H^{k+1}(\Omega_x)}, |||\partial_{tt} g|||, \langle||g||^2_{H^{k+1}(\Omega_x)}\rangle,
\langle||\partial_t g||^2_{H^{k+1}(\Omega_x)}\rangle$$ over $t\in[0,T]$.
Different occurrences of $C, C_*$ could take different values. Standard notations for Sobolev spaces $H^{k+1}(\Omega_x)$ as well as their norms $||\cdot||_{H^{k+1}(\Omega_x)}$ are used in this paper \cite{ciarlet:book}.

Recall that at $t=0$, the proposed methods are initialized through $\rho_h^0=\pi_h\rho^0$ and $g_h^0=\pi_h g^0$.
With this and Lemma \ref{lem:0}, we have $\langle \xi_g^n\rangle=\langle\Pi_g g^n\rangle -\langle g_h^n\rangle=\Pi_g\langle g^n\rangle=0 $, hence
\begin{equation}\label{eq:a:xign}
\langle \xi_g^n\rangle=0, \;\;\;\;\forall n.
\end{equation}
  In next Theorem, we will state the main error estimate results, and their proofs will be given step by step in Sections \ref{sec:3.2.1}-\ref{sec:3.2.5}.

\begin{thm}[Error estimate]  When the DG-IMEX1 method \eqref{eq:FDG:1T} is applied to the kinetic transport equation \eqref{f-JH} in its micro-macro decomposition formulation \eqref{eq:mmd:LJ},
%with the initialization $\rho_h^0=\Pi_\rho \rho^0$ and $g_h^0=\Pi_g g^0$,
the following error estimates hold:
\begin{enumerate}
\item[(1)] with any of the alternating flux in \eqref{eq:flux},
\begin{align}
&||\rho^n-\rho_h^{n}||^2+\eps^2|||g^{n-1}-g_h^{n-1}|||^2 \notag\\
&\leq C_*\left((1+\eps^2)h^{2k+2}+\Dt^2 +\frac{1}{1-\sigma}((1+\eps^4)\Dt^2+h^{2k+2}+\eps h^{2k+1})\right);
\end{align}
\item[(2)] with any of the central flux in \eqref{eq:flux},
\begin{align}
&||\rho^n-\rho_h^{n}||^2+\eps^2|||g^{n-1}-g_h^{n-1}|||^2 \notag\\
&\leq C_*\left(\eps^2 h^{2k+2}+h^{2k}+\Dt^2 +\frac{1}{1-\sigma} ((1+\eps^4)\Dt^2+h^{2k}+\eps h^{2k+1})\right)
\end{align}
\end{enumerate}
for $n: n\Dt \leq T$ under the condition $\Dt\leq \sigma\Dt_{\textrm{stab}}$ and $\Dt<\frac{1}{2}$. Here $\sigma$ is any constant in $(0,1)$.
\label{thm:main}
\end{thm}

The error estimates are obtained as long as the time step is no larger than that required for numerical stability. We further summarize the established spatial accuracy orders in  Table \ref{Table:convg}. Numerically, higher than theoretical convergence rates can be observed in some cases (see \cite{JLQX_numerical}).
% and Section \ref{sec:numerical}.
%
\begin{table}
\begin{center}
 \caption{Spatial accuracy orders established by the error estimates.}
 \vspace{0.2cm}
\begin{tabular}{|c|c|c|}
\hline
 &alternating &central \\
\hline
$\eps=O(h)$ &$k+1$& $k$\\\hline
$\eps>> h$& $k+\frac{1}{2}$ & $k$\\
 \hline
\end{tabular}
\label{Table:convg}
\end{center}
\end{table}
\begin{rem} What established here are a priori error estimates, and the constant $C_*$ depends on exact solutions hence possibly on $\eps$.  For any test case  where $\sup_\eps C_* <\infty$, the error estimates in Theorem \ref{thm:main} hold uniformly with respect to $\eps$.
\end{rem}

\subsubsection{Local truncation errors and error equations}
\label{sec:3.2.1}

Using the consistency of the DG spatial discretization, local truncation errors from the $n$-th step  temporal discretization, denoted as $\tau_\rho^n(\cdot) \in U_h^k$ and $\tau_g^n(\cdot,v)\in U_h^k$, are defined as follows. For any $\phi(\cdot), \psi(\cdot, v)\in U_h^k$,
\begin{subequations}
\label{eq:trun}
\begin{align}
&\left(\frac{\rho^{n+1}-\rho^n}{\Dt}, \phi\right) + a_h(g^n, \phi)=(\tau_\rho^n, \phi), \label{eq:trun:r}\\
&\left(\frac{g^{n+1}-g^n}{\Dt}, \psi\right)+ \frac{1}{\eps} b_{h,v}(g^n, \psi)-\frac{v}{\eps^2} d_h(\rho^{n+1}, \psi) =
 -\frac{1}{\eps^2}(g^{n+1}, \psi)+(\tau_g^n, \psi).
 \label{eq:trun:g}
\end{align}
\end{subequations}
\begin{lem} The following estimates hold for the local truncation errors,
\begin{align}
&||\tau_\rho^n||\leq \frac{\Dt}{\sqrt{3}}\max_{t\in[0,T]}||\partial_{tt} \rho(\cdot,t)||,\\
&|||\eps^2\tau_g^n|||\leq \eps^2\frac{\Dt}{\sqrt{3}} \max_{t\in[0,T]}|||\partial_{tt} g|||+\Dt \max_{t\in[0,T]}(|||\partial_t g|||+{\langle v^2\rangle}^{1/2}||\partial_{xt}\rho||).
\end{align}
\label{lem:trun}
\end{lem}
\begin{proof}
For any $\phi\in U_h^k$, using the definition of the local truncation errors
 and  integration by parts on each mesh element, we have
\begin{align*}
(\tau_\rho^n, \phi)&=\left(\frac{\rho^{n+1}-\rho^n}{\Dt}, \phi\right) + a_h(g^n, \phi)
                    =\left(\frac{\rho^{n+1}-\rho^n}{\Dt}-\partial_t \rho^n, \phi\right) + (\partial_t \rho^n, \phi)+ a_h(g^n, \phi)\\
&=\left(\frac{\rho^{n+1}-\rho^n}{\Dt}-\partial_t \rho^n, \phi\right) + ((\partial_t \rho +\partial_x \langle v g\rangle|_{t=t^n}, \phi)
=\left(\frac{\rho^{n+1}-\rho^n}{\Dt}-\partial_t \rho^n, \phi\right).
\end{align*}
The third equality is due to the consistency of the numerical flux, therefore the consistency of the spatial discretization. Now
\begin{align*}
||\tau_\rho^n||&=\max_{0\ne\phi\in U_h^k}\frac{(\tau_\rho^n, \phi)}{||\phi||}
\leq ||\frac{\rho^{n+1}-\rho^n}{\Dt}-\partial_t \rho^n||
=\frac{1}{\Dt}||\int_{t^n}^{t^{n+1}} (t^{n+1}-t)\partial_{tt} \rho(x,t)dt||\\
&\leq \frac{1}{\Dt}\left( \int_{t^n}^{t^{n+1}} (t^{n+1}-t)^2dt \int_{\Omega_x}\int_{t^n}^{t^{n+1}} (\partial_{tt} \rho(x,t))^2 dt dx   \right)^{1/2}
%\\
%&=\frac{1}{\Dt}\left( \frac{\Dt^3}{3}\int_{t^n}^{t^{n+1}} ||\partial_{tt} \rho(\cdot,t)||^2 dt\right)^{1/2}
\leq \frac{\Dt}{\sqrt{3}}\max_{t\in[0,T]}||\partial_{tt} \rho(\cdot,t)||.
\end{align*}
Similarly, for the truncation error in equation for $g$, $\forall\psi\in U_h^k$,
\begin{align*}
(\eps^2\tau_g^n, \psi)
&=\eps^2\left(\frac{g^{n+1}-g^n}{\Dt}, \psi\right)+ \eps b_{h,v}(g^n, \psi)-v d_h(\rho^{n+1}, \psi) +(g^{n+1}, \psi)\\
&=\eps^2\left(\partial_t g^n, \psi\right)+ \eps b_{h,v}(g^n, \psi)-v d_h(\rho^n, \psi) +(g^n, \psi)\\
&\;\;\;\;\;+\eps^2\left(\frac{g^{n+1}-g^n}{\Dt}-\partial_t g^n, \psi\right)-v d_h(\rho^{n+1}-\rho^n, \psi) +(g^{n+1}-g^n, \psi)\\
&=\eps^2\left((\partial_t g +\frac{1}{\eps^2}v\partial_x\rho +\frac{1}{\eps}\{\mathbf{I}-\Pi\}(v\partial_xg)+ \frac{1}{\varepsilon^2} g)|_{t=t^n}, \psi\right)\\
&\;\;\;\;\;+\eps^2\left(\frac{g^{n+1}-g^n}{\Dt}-\partial_t g^n, \psi\right) + (v \partial_x (\rho^{n+1}-\rho^n), \psi) +(g^{n+1}-g^n, \psi)\\
&=\eps^2\left(\frac{g^{n+1}-g^n}{\Dt}-\partial_t g^n, \psi\right) + (v\partial_x(\rho^{n+1}-\rho^n), \psi) +(g^{n+1}-g^n, \psi).
\end{align*}
Note $||\eps^2 \tau_g^n ||=\max_{0\ne\psi\in U_h^k}\frac{(\eps^2\tau_g^n, \psi)}{||\psi||}$, then
\begin{align*}
|||\eps^2 \tau_g^n|||&\leq |||\eps^2 (\frac{g^{n+1}-g^n}{\Dt}-\partial_t g^n) +v\partial_x(\rho^{n+1}-\rho^n)+(g^{n+1}-g^n)|||\\
&\leq \eps^2||| \frac{g^{n+1}-g^n}{\Dt}-\partial_t g^n||| +|||v\partial_x(\rho^{n+1}-\rho^n)|||+|||(g^{n+1}-g^n)|||\\
&\leq \eps^2\frac{\Dt}{\sqrt{3}} \max_{t\in[0,T]}|||\partial_{tt} g|||+\Dt \max_{t\in[0,T]}(|||\partial_t g|||+\langle v^2\rangle ^{1/2}||\partial_{xt}\rho||).
\end{align*}
\end{proof}

Subtracting the numerical scheme \eqref{eq:FDG:1T} from \eqref{eq:trun} and using the linearity, one gets the following error equations,
\begin{subequations}
\label{eq:err}
\begin{align}
\left(\frac{\xi_\rho^{n+1}-\xi_\rho^n}{\Dt}, \phi\right) & + a_h(\xi_g^n, \phi)=\left(\frac{\eta_\rho^{n+1}-\eta_\rho^n}{\Dt}, \phi\right) + a_h(\eta_g^n, \phi)+(\tau_\rho^n, \phi), \label{eq:err:r}\\
\left(\frac{\xi_g^{n+1}-\xi_g^n}{\Dt}, \psi\right) &+ \frac{1}{\eps} b_{h,v}(\xi_g^n, \psi)-\frac{v}{\eps^2} d_h(\xi_\rho^{n+1}, \psi) +\frac{1}{\eps^2}(\xi_g^{n+1}, \psi)\label{eq:err:g}\\
& =\left(\frac{\eta_g^{n+1}-\eta_g^n}{\Dt}, \psi\right) + \frac{1}{\eps} b_{h,v}(\eta_g^n, \psi)-\frac{v}{\eps^2} d_h(\eta_\rho^{n+1}, \psi) +\frac{1}{\eps^2}(\eta_g^{n+1}, \psi)+(\tau_g^n, \psi)\notag
\end{align}
\end{subequations}
for any test functions $\phi(\cdot), \psi(\cdot, v)\in U_h^k$. We now take $\phi=\xi_\rho^{n+1}$ in \eqref{eq:err:r}. In addition, we take $\psi=\eps^2\xi_g^{n+1}$ in \eqref{eq:err:g}, integrate in $v$, and shift the index $n$ to $n-1$.  The resulting two equations are summed up and give $\textrm{LHS}=\textrm{RHS}$,
where
\begin{align}
\textrm{LHS}=&\Xi_\eps(\xi_\rho^{n+1}, \xi_\rho^n, \xi_g^n, \xi_g^{n-1})+|||\xi_g^n|||^2, \label{eq:err:lhs}\\
\textrm{RHS}=&\left(\frac{\eta_\rho^{n+1}-\eta_\rho^n}{\Dt}, \xi_\rho^{n+1}\right) + a_h(\eta_g^n, \xi_\rho^{n+1})+(\tau_\rho^n, \xi_\rho^{n+1})\label{eq:err:rhs}\\
             &+\eps^2\langle\left(\frac{\eta_g^n-\eta_g^{n-1}}{\Dt}, \xi_g^n\right)\rangle + \eps\langle b_{h,v}(\eta_g^{n-1}, \xi_g^n)\rangle-\langle v d_h(\eta_\rho^n, \xi_g^n)\rangle +\langle(\eta_g^n, \xi_g^n)\rangle+\eps^2\langle(\tau_g^{n-1}, \xi_g^n)\rangle.\notag
\end{align}
 Now we apply Lemma \ref{lem:1} with $\phi=\xi_\rho$, $\psi=\xi_g$, and $m=n-1$, and get
\begin{align}
\textrm{LHS}&\geq\frac{1}{2\Dt}(\mE_{n+1}-\mE_n)\label{LHS:bound}\\
%\left((||\xi_\rho^{n+1}||^2+\eps^2|||\xi_g^{n}|||^2)-(||\xi_\rho^n||^2+\eps^2|||\xi_g^{n-1}|||^2)\right)\notag\\
 &+\left\{\begin{array}{ll}
 \left(\eps-\alpha_2\frac{\Dt}{h}\right)\langle\frac{|v|}{2}\sum_i [\xi_g^n]_\iL^2\rangle
    +\left(1-\alpha_1 \frac{\Dt}{h^2}\right) |||\xi_g^n|||^2,& \textrm{for}\;k\geq 1,\\
 \left(\eps-\frac{\alpha_2}{2}\frac{\Dt}{h}\right)\langle\frac{|v|}{2}\sum_i [\xi_g^n]_\iL^2\rangle+|||\xi_g^n|||^2, &\textrm{for}\;k=0.
 \end{array}
 \right.
\notag
\end{align}
Next we want to estimate \textrm{RHS}. This will be proceeded for the proposed scheme with the alternating and the central flux \eqref{eq:flux}, respectively.

\subsubsection{To estimate \textrm{RHS} \eqref{eq:err:rhs} with the alternating flux}
\label{sec:3.2.2}

Without loss of generality, we consider using
the alternating flux $\widehat{\langle vg\rangle}={\langle vg\rangle}^-$, $\hat{\rho}= {\rho}^+$
in the proposed scheme \eqref{eq:FDG:1T}. For this case, we take $\Pi_\rho=\pi_h^+$ and $\Pi_g=\pi_h^-$ in error estimates. With such choices and $\df_x\xi_\rho^{n+1}, \df_x\langle v \xi_g^n \rangle(\cdot,v)\in U_h^{k-1}$, one has  $\int_{I_i}  \langle v\eta_g^n \rangle \df_x\xi_\rho^{n+1} dx=\int_{I_i} \eta_\rho^n \df_x \langle v \xi_g^n \rangle dx=0$,
and $\langle v\eta_g^n\rangle^-_{\iL}=(\eta_\rho^n)^+_\iL=0, \forall i$, therefore
\beq
a_h(\eta_g^n, \xi_\rho^{n+1})=\langle v d_h(\eta_\rho^n, \xi_g^n)\rangle =0.
\label{RHS:bound-1}
\eeq
We now turn to the term containing $ b_{h,v}$ in \eqref{eq:err:rhs}.
With $\langle \xi_g^n\rangle=0, \forall n$ in \eqref{eq:a:xign}, and the choice of $\Pi_g$, we have
\begin{align*}
&\langle b_{h,v}(\eta_g^{n-1}, \xi_g^n)\rangle=\langle  (\mD_{h,v}(\eta_g^{n-1}; v), \xi_g^n) \rangle\\
&= - \langle \sum_i\left(\int_{I_i} v \eta_g^{n-1} \df_x \xi_g^n dx\right) + \sum_i \widetilde{(v\eta_g^{n-1})}_\iL[\xi_g^n]_\iL \rangle= - \langle  \sum_i \widetilde{(v\eta_g^{n-1})}_\iL[\xi_g^n]_\iL \rangle\\
&= - \langle \sum_i (v\eta_g^{n-1})^-_\iL[\xi_g^n]_\iL \rangle_{+}- \langle \sum_i (v\eta_g^{n-1})^+_\iL[\xi_g^n]_\iL \rangle_{-}= -  \langle \sum_i (v\eta_g^{n-1})^+_\iL[\xi_g^n]_\iL \rangle_{-}.
\end{align*}
This can be further estimated as below, with any $\sigma\in(0,1)$
\begin{align}
|\langle b_{h,v}(\eta_g^{n-1}, \xi_g^n)\rangle|
&\leq (1-\sigma) \langle \frac{|v|}{2}\sum_i  [\xi_g^n]_\iL^2 \rangle_{-}
+\frac{1}{4(1-\sigma)} \langle 2|v| \sum_i ((\eta_g^{n-1})^+_\iL)^2  \rangle_{-}\notag\\
&\leq (1-\sigma)  \langle \frac{|v|}{2}\sum_i [\xi_g^n]_\iL^2 \rangle
+\frac{C_*}{(1-\sigma)} h^{2k+1}.
\label{RHS:bound-2}
\end{align}
The estimate in \eqref{eq:pi:1} is
used to get the last inequality. We further note that
\begin{align}
||\frac{\eta_\rho^{n+1}-\eta_\rho^n}{\Dt}||^2
=&\int_{\Omega_x}|\frac{1}{\Dt}\int_{t^n}^{t^{n+1}} (I-\Pi_\rho)\partial_t \rho(x,s)ds|^2 dx\notag\\
\leq & \frac{1}{\Dt}\int_{t^n}^{t^{n+1}} ||(I-\Pi_\rho)\partial_t\rho(\cdot,s)||^2ds\notag\\
\leq& C\max_{t\in[0,T]}||\partial_t\rho(\cdot,t)||^2_{H^{k+1}(\Omega_x)}\;h^{2k+2}=C_* h^{2k+2},
\label{RHS:bound-2.5}
\end{align}
and similarly
\begin{align}
|||\frac{\eta_g^n-\eta_g^{n-1}}{\Dt}|||^2
&\leq C \max_{t\in[0,T]}\langle ||\partial_t g(\cdot,t)||^2_{H^{k+1}(\Omega_x)}\rangle \;h^{2k+2}=C_* h^{2k+2}.
\end{align}
With these and the estimates on the truncation errors, we have
\begin{align}
&|\left(\frac{\eta_\rho^{n+1}-\eta_\rho^n}{\Dt}, \xi_\rho^{n+1}\right) +(\tau_\rho^n, \xi_\rho^{n+1})|
\notag\\
&\leq ||\xi_\rho^{n+1}||^2
+\frac{1}{4}||\frac{\eta_\rho^{n+1}-\eta_\rho^n}{\Dt}+\tau_\rho^n||^2
\leq ||\xi_\rho^{n+1}||^2
+C_* (h^{2k+2}+\Dt^2),
\label{RHS:bound-3}
\end{align}
and
\begin{align}
|\eps^2 &\langle\left(\frac{\eta_g^n-\eta_g^{n-1}}{\Dt}, \xi_g^n\right)\rangle +
\langle(\eta_g^n, \xi_g^n)\rangle+\eps^2\langle(\tau_g^{n-1}, \xi_g^n)\rangle|\notag\\
\leq &\eps^2|||\xi_g^n|||^2+\frac{1}{4}\eps^2 |||\frac{\eta_g^n-\eta_g^{n-1}}{\Dt}|||^2+(1-\sigma)  |||\xi_g^n|||^2+\frac{1}{4(1-\sigma) }
(|||\eps^2\tau_g^{n-1}|||^2+|||\eta_g^n|||^2)\notag\\
\leq &\eps^2|||\xi_g^n|||^2+(1-\sigma)  |||\xi_g^n|||^2
+\eps^2 C_* h^{2k+2}+\frac{C_*}{(1-\sigma) }\left( (1+\eps^4)\Dt^2+h^{2k+2}\right).
\label{RHS:bound-4}
\end{align}
Now we combine the estimates in \eqref{RHS:bound-1}, \eqref{RHS:bound-2}, \eqref{RHS:bound-3}, \eqref{RHS:bound-4},
and get
\beq
|\textrm{RHS}|\leq \mE_{n+1}
+(1-\sigma) |||\xi_g^n|||^2+(1-\sigma) \eps \langle \frac{|v|}{2}\sum_i  [\xi_g^n]_\iL^2 \rangle+C_*\tau,
\label{RHS:bound}
\eeq
where
\beq
\tau= (1+\eps^2)h^{2k+2}+\Dt^2 +\frac{1}{1-\sigma}((1+\eps^4)\Dt^2+h^{2k+2}+\eps h^{2k+1}).
\label{eq:tau:1}
\eeq

\subsubsection{To estimate \textrm{RHS} \eqref{eq:err:rhs}  with the central flux}
\label{sec:3.2.3}

Now we consider the proposed scheme \eqref{eq:FDG:1T} with the central flux $\widehat{\langle vg\rangle}=\{{\langle vg\rangle}\}, \hat{\rho}=\{\rho\}$. For this case, we take $\Pi_\rho=\pi_h$ and $\Pi_g=\pi_h$. With such choices,
there is
\begin{align*}
\textrm{RHS}=&a_h(\eta_g^n, \xi_\rho^{n+1})+(\tau_\rho^n, \xi_\rho^{n+1})
             + \eps\langle b_{h,v}(\eta_g^{n-1}, \xi_g^n)\rangle-\langle v d_h(\eta_\rho^n, \xi_g^n)\rangle +\eps^2\langle(\tau_g^{n-1}, \xi_g^n)\rangle\\
             =& (\tau_\rho^n, \xi_\rho^{n+1})+\eps^2\langle(\tau_g^{n-1}, \xi_g^n)\rangle
             -\sum_i\left(\{\langle v\eta_g^n\rangle\}[\xi_\rho^{n+1}]
             +\eps\langle \widetilde{(v\eta_g^{n-1})}[\xi_g^n]\rangle
             +\langle v \{\eta_\rho^n\}[\xi_g^n]\rangle\right)_\iL.
\end{align*}
Moreover, one can use the estimates in \eqref{eq:pi:1},
the estimates for the truncation errors, the inverse inequality \eqref{eq:inv:1}, and have
\begin{align*}
&|(\tau_\rho^n, \xi_\rho^{n+1})+\eps^2\langle(\tau_g^{n-1}, \xi_g^n)\rangle|
\leq \frac{1}{2}||\xi_\rho^{n+1}||^2+\frac{1-\sigma}{2}|||\xi_g^n|||^2+C_*\Dt^2+\frac{C_*}{1-\sigma}(1+\eps^4)\Dt^2,\\
&|\sum_i(\{\langle v\eta_g^n\rangle\}[\xi_\rho^{n+1}])_\iL|
 \leq \frac{h}{4C_{\textrm{inv}}} \sum_i[\xi_\rho^{n+1}]^2_\iL
 +\frac{C_{\textrm{inv}}}{ h}\sum_i \{\langle v\eta_g^n\rangle\}^2_\iL
 \leq \frac{1}{2}||\xi_\rho^{n+1}||^2 +C_* h^{2k},\\
 %+\delta(k)-1},\\
&|\eps\sum_i\langle \widetilde{(v\eta_g^{n-1})}[\xi_g^n]\rangle_\iL|
\leq (1-\sigma) \eps \langle \frac{|v|}{2}\sum_i  [\xi_g^n]_\iL^2 \rangle+\frac{C_*}{1-\sigma}\eps h^{2k+1},\\
%\delta(k)},\\
&\sum_i \langle v \{\eta_\rho^n\}[\xi_g^n]\rangle_\iL
\leq \frac{1-\sigma}{2}|||\xi_g^n|||^2 + \frac{C_*}{1-\sigma}h^{2k},
%+\delta(k)-1},
\end{align*}
therefore RHS can be estimated as \eqref{RHS:bound}
where
\beq
\tau= h^{2k}+\Dt^2 +\frac{1}{1-\sigma} ((1+\eps^4)\Dt^2+h^{2k}+\eps h^{2k+1}).
%\tau= h^{2k+\delta(k)-1}+\Dt^2 +\frac{1}{1-\sigma} ((1+\eps^4)\Dt^2+h^{2k+\delta(k)-1}+\eps h^{2k+\delta(k)}).
\label{eq:tau:2}
\eeq

\subsubsection{To estimate ${\mathcal E}_1$}
\label{sec:3.2.4}

With the discrete energy defined in the analysis, we also need to estimate ${\mathcal E}_1$. To achieve this, we start with \eqref{eq:err:r}, $n=0$, and have
 %together with $\xi_\rho^0=\xi_g^0=0$, and have
\begin{align*}
||\xi_\rho^1||=\sup_{0\ne\phi\in U_h^k}\frac{(\xi_\rho^1, \phi)}{||\phi||}
&=\sup_{0\ne\phi\in U_h^k}\frac{\left(\eta_\rho^1+e_\rho^0, \phi\right) - \Dt a_h(e_g^0, \phi)+
\Dt(\tau_\rho^0,\phi)}{||\phi||}\\
&\leq ||\eta_g^1||+||e_\rho^0||+\Dt\left (||\tau_\rho^0||+ \sup_{0\ne\phi\in U_h^k}\frac{a_h(e_g^0, \phi)}{||\phi||}\right)\;.
\end{align*}
The first two terms can be estimated based on the property of the projections in \eqref{eq:pi:1}, that is,
$||\eta_\rho^1||=||\rho^1-\Pi_\rho\rho^1||=C_* h^{k+1}$ and $||e_\rho^0||=||\rho^0-\pi_h\rho^0||=C_* h^{k+1}$. From Lemma \ref{lem:trun}, the truncation error $\tau_\rho^0$ can be controlled by $||\tau_\rho^0||\leq C_*\Dt$. To estimate the last term, we use $e_g^0=g^0-\pi_h g^0$  and \eqref{eq:pi:1}, and have
\begin{align*}
|a_h(e_g^0, \phi)|=  |\sum_i \widehat{\langle v e_g^0\rangle}_{\iL} [\testR]_\iL|
\leq \left(h^{-1} \sum_i \widehat{\langle v e_g^0\rangle}_{\iL}^2 \right)^{1/2} \left(h \sum_i [\testR]_\iL^2\right)^{1/2}
\leq C_* h^k ||\phi||,
\end{align*}
for any $\phi\in U_h^k$. Note that $\Dt\leq\sigma\Dt_{stab}$ implies $\Dt=C(h^2+\eps h)$, with this, we finally have
\begin{align*}
{\mathcal E}_1&=||\xi_\rho^1||^2+\eps^2|||\xi_g^0|||^2=||\xi_\rho^1||^2+\eps^2||\Pi_g g^0-g^0+g^0-\pi_h g^0||^2\\
&\leq C_*\left(\Dt^2 h^{2k}+(1+\eps^2)h^{2k+2}+\Dt^4\right)\leq C_*\left((1+\eps^2)h^{2k+2}+\Dt^4\right).
%\label{eq:theta1:bound}
\end{align*}

\subsubsection{The final step for the error estimates}
\label{sec:3.2.5}

We now combine the bounds of \textrm{LHS} and \textrm{RHS} in \eqref{LHS:bound}, \eqref{RHS:bound} with \eqref{eq:tau:1} for the alternating flux, or  with \eqref{eq:tau:2} for the central flux. Together with \eqref{eq:thm1:4}, we have
\begin{align}
\frac{1}{2\Dt}(\mE_{n+1}-\mE_n) \leq \mE_{n+1} +\hat{\gamma} |||\xi_g^n|||^2+C_*\tau
\end{align}
with
\begin{align}
\hat{\gamma}=
\left\{\begin{array}{ll}
\max\left(-\sigma\eps+ \alpha_2 \frac{\Dt}{h}, 0\right)\frac{\alpha_3}{h}+\left(-\sigma+ \alpha_1 \frac{\Dt}{h^2}\right)&\textrm{for}\; k\geq 1\\
\max\left(-\sigma\eps+ \frac{\alpha_2}{2} \frac{\Dt}{h}, 0\right)\frac{\alpha_3}{h}-\sigma &\textrm{for}\; k=0\;.
\end{array}
\right.
\end{align}

As long as $\hat{\gamma}\leq 0$, that is,
\begin{equation}
\Dt \leq \sigma\Dt_{stab},
\end{equation}
where $\Dt_{stab}$ comes from the time step constraint in \eqref{eq:cfl} for numerical stability,
then
\begin{equation}
(1-2\Dt) \mE_{n+1}\leq \mE_n +(2\Dt) C_*\tau.
\end{equation}
Assume $\Dt < \frac{1}{2}$.
Define $\Theta_n=\mE_n(1-2\Dt)^n$, then
$
\Theta_{n+1}\leq \Theta_n + 2\Dt (1-2\Dt)^n C_* \tau.
$
 With mathematical induction, and the estimate of ${\mathcal E}_1$ in Section \ref{sec:3.2.4}, one gets
\begin{equation}
\Theta_n\leq 2\Dt  C_* \tau ((1-2\Dt)+\cdots +(1-2\Dt)^{n-1})+\Theta_1\leq  (C_* \tau+{\mathcal E }_1) (1-2\Dt)\leq C_* \tau,
\end{equation}
that is
$\mE_n(1-2\Dt)^n\leq C_* \tau.$
Moreover, with $n: \Dt n\leq T$, there is
\begin{equation}
\mE_n\leq (1-2\Dt)^{-n} C_* \tau\leq e^{2\Dt n} C_* \tau\leq C_*e^{2T} \tau=C_*\tau.
\end{equation}

On the other hand, the properties of the projection operators \eqref{eq:pi:1} ensure $||\eta_\rho^n||=C_* h^{k+1}$ and $|||\eta_g^{n-1}|||=C_* h^{k+1}$, therefore
\begin{align}
||e_\rho^n||^2+\eps^2|||e_g^{n-1}|||^2
\leq 2(||\eta_\rho^n||^2+\eps^2|||\eta_g^{n-1}|||^2+\mE_n) \leq C_*(\tau+(1+\eps^2)h^{2k+2}).
\end{align}
We now can conclude the main error estimates in Theorem \ref{thm:main} by further utilizing the forms of $\tau$ in \eqref{eq:tau:1} and \eqref{eq:tau:2}.

%% file: analysis3.tex
\subsection{Rigorous asymptotic analysis}
\label{sec:3.3}

In \cite{JLQX_numerical}, a formal asymptotic analysis was performed for the proposed methods, showing that when $\eps\rightarrow 0$, the limiting schemes are
consistent discretizations for the limiting heat equation. In this section, we want to establish this asymptotic preserving property rigorously for DG-IMEX1 using tools from functional analysis.

To explicitly indicate how the numerical solutions depend on the mesh parameter $\Dt$, $h$, and on $\eps$, we use $\rho_{\eps, \Dt, h}^n$ and $g_{\eps, \Dt, h}^n$ to denote $\rho_h^n$ and $g_h^n$ in this section, unless otherwise specified.
For the initial condition $\rho_\eps (x, t=0)$ and $g_\eps(x, v, t=0)$, two assumptions are made which are  mild and reasonable.
%
%%%%%%%%%%%%%%%%
% Assumption: Initial
%%%%%%%%%%%%%%%

\begin{itemize}
\item[(${\mathcal A}1$)] At $t=0$,
\begin{equation}
\rho_\eps \rightharpoonup \rho_0, \quad \langle vg_\eps \rangle \rightharpoonup \langle vg_0 \rangle\;\;\textrm{in} \; L^2(\Omega_x),\quad \textrm{as}\; \eps\rightarrow 0.
\end{equation}
 Here ``$\rightharpoonup$'' stands for weak convergence.
\item[(${\mathcal A}2$)]
\begin{equation}
\sup_\eps (||\rho_\eps||)|_{t=0}<\infty,\qquad \sup_\eps(|||g_\eps|||)|_{t=0}<\infty.
\end{equation}
\end{itemize}

In the following analysis, the index $k$ for the discrete space $U_h^k$, hence in the numerical method, is fixed.  Let $\{\Psi_j\}_{j=1}^{N_k}$ denote an orthonormal basis of $U_h^k$ with respect to the standard $L^2$ inner product. We define
\begin{equation}
\label{eq:init:limit}
\rho_{\Dt, h}^0=\pi_h\rho_0|_{t=0},\qquad q_{\Dt, h}^0=\pi_h \langle v g_0\rangle|_{t=0},
\end{equation}
 and also denote $q^n_{\eps, \Dt, h}=\langle v g^n_{\eps, \Dt, h}\rangle$, $\forall n$. Below we will discuss some properties of the numerical solution at $t=t^n$ with $n=0, 1$, before turning to the main result in Theorem  \ref{thm:asymptotic}.

%%%%%%%%%%%%%%%%
% Lemma: Initial
%%%%%%%%%%%%%%%
\begin{lem}
\label{lem:init}
Under the assumptions (${\mathcal A}1$) and (${\mathcal A}2$), we have
\begin{itemize}
\item[(i)] $\lim_{\eps\rightarrow 0}\rho^0_{\eps, \Dt, h}= \rho^0_{\Dt, h}$, and $\lim_{\eps\rightarrow 0}q^0_{\eps, \Dt, h}= q^0_{\Dt, h}$. The convergence is in any norm.
\item[(ii)] $\sup_\eps||\rho_{\eps, \Dt, h}^{1}||<\infty$, and  $\sup_\eps|||g_{\eps, \Dt, h}^{0}|||<\infty$.
\end{itemize}
\end{lem}
\begin{proof}

First note that $U^k_h$ is finite dimensional, then a sequence in $U^k_h$, if converges, will converge in any norm.
To prove (i), based on assumption (${\mathcal A}1$), we have $\lim_{\eps\rightarrow\infty}(\rho_\eps, \psi)=(\rho_0, \psi)$, $\forall\psi\in L^2(\Omega_x)$ at $t=0$, therefore  as $\eps\rightarrow 0$,
\begin{align*}
\rho^0_{\eps, \Dt, h}=\pi_h{\rho_\eps}|_{t=0}=\sum_{j=1}^{N_k}(\rho_\eps|_{t=0}, \Psi_j) \Psi_j
\rightarrow \sum_{j=1}^{N_k}(\rho_0|_{t=0}, \Psi_j) \Psi_j=\pi_h\rho_0|_{t=0}=\rho^0_{\Dt, h}.
\end{align*}
Similarly, the second half of (i) can be proved.

To prove (ii), recall that $g^0_{\eps, \Dt, h}=\pi_h g_\eps|_{t=0}$, then
\begin{align*}
|||g^0_{\eps, \Dt, h}|||^2=\langle ||g^0_{\eps, \Dt, h}||^2\rangle=\langle \sum_{j=1}^{N_k}(g_\eps|_{t=0}, \Psi_j)^2\rangle
\leq (|||g_\eps|||^2)|_{t=0}\sum_{j=1}^{N_k}||\Psi_j||^2.
\end{align*}
Now with assumption (${\mathcal A}2$), we have $\sup_\eps|||g^0_{\eps, \Dt, h}|||<\infty$. Similarly, one can show
that
\begin{equation}
\label{lem:init:1}
\sup_\eps||\rho^0_{\eps, \Dt, h}||\leq (||\rho_\eps||)|_{t=0} \left(\sum_{j=1}^{N_k}||\Psi_j||^2\right)^{1/2}<\infty.
\end{equation}
What remains is to establish the boundedness of $\sup_\eps||\rho_{\eps, \Dt, h}^{1}||$.

From \eqref{eq:FDG:1T:r} with $n=0$,
\begin{equation}
\label{lem:init:2}
||\rho^1_{\eps,\Dt, h}||\leq ||\rho^0_{\eps,\Dt, h}||+\Dt\sup_{0\ne\phi\in U_h^k}\frac{a_h(g^0_{\eps,\Dt, h}, \phi)}{||\phi||}.
\end{equation}
Without loss of generality, in $a_h(\cdot,\cdot)$ defined in \eqref{eq:ah}, we consider $\widehat{\langle vg\rangle}=\langle vg\rangle^-$. First we have
\begin{align}
\label{lem:init:3}
||\langle v g_{\eps,\Dt, h}^0\rangle||^2&=||\sum_{j=1}^{N_k} (\langle v g_\eps\rangle|_{t=0},\Psi_j)\Psi_j||^2=\sum_{j=1}^{N_k}(\langle v g_\eps\rangle|_{t=0},\Psi_j)^2\notag\\
&\leq (||\langle v g_\eps\rangle||^2)|_{t=0}\sum_{j=1}^{N_k}||\Psi_j||^2
\leq \langle v^2\rangle (||| g_\eps|||^2)|_{t=0} \sum_{j=1}^{N_k}||\Psi_j||^2,
\end{align}
and
\begin{align}
\label{lem:init:4}
|\langle v g_{\eps,\Dt, h}^0\rangle^-_{\iL}|&=|\sum_{j=1}^{N_k} (\langle v g_\eps\rangle|_{t=0},\Psi_j)\Psi_j(x^-_\iL)|\notag\\
&\leq \langle v^2\rangle^{1/2} (||| g_\eps|||)|_{t=0} \sum_{j=1}^{N_k}||\Psi_j|| |\Psi(x^-_\iL)|.
\end{align}
Based on the definition of $a_h(\cdot, \cdot)$, and inverse inequalities in \eqref{eq:inv}, we further get
\begin{align}
\label{lem:init:5}
a_h(g^0_{\eps,\Dt, h}, \phi)&=
-( \langle v g^0_{\eps,\Dt, h}\rangle, \df_x\testR) - \sum_i \widehat{\langle vg^0_{\eps,\Dt, h}\rangle}_{\iL} [\testR]_\iL\notag\\
&\leq ||\langle v g^0_{\eps,\Dt, h}\rangle||\;||\df_x\testR||+\sum_i |\widehat{\langle vg^0_{\eps,\Dt, h}\rangle}_{\iL}| \; |[\testR]_\iL|\notag\\
&\leq \frac{C(k)}{h} ||\phi||  \left( ||\langle v g^0_{\eps,\Dt, h}\rangle||+ \left(h\sum_i (\langle v g_{\eps,\Dt, h}^0\rangle^-_{\iL}  )^2\right)^{1/2}\right).
\end{align}
Now we can combine \eqref{lem:init:1}-\eqref{lem:init:5} as well as the boundedness assumption (${\mathcal A}2$), and conclude
\begin{equation*}
\sup_\eps||\rho^1_{\eps,\Dt, h}||\leq C(k, \Dt, h, \langle v^2\rangle) \left(\sup_\eps||\rho_\eps||+\sup_\eps|||g_\eps|||\right)|_{t=0} <\infty.
\end{equation*}
\end{proof}
%
%%%%%%%%%%%%%%%%
% Theorem:
%%%%%%%%%%%%%%%
\begin{thm}
\label{thm:asymptotic}
Let $c_0$ be any fixed positive constant, $c_0\in(0,1)$. Under the assumptions (${\mathcal A}1$) and (${\mathcal A}2$) on the initial data, and under the following condition on the time step $\Dt$,
 \beq
 \label{eq:Dt0}
     \Dt <\Dt_{stab,c_0}=\frac{h}{\alpha_1 +\alpha_2\alpha_3} ((1-c_0)h+\min(\eps, \frac{(1-c_0)\alpha_2h}{\alpha_1})\alpha_3),
 \eeq
with $\alpha_i, i=1, 2, 3$ defined in \eqref{eq:alpha}, we have
\beq
\label{eq:aa:5}
 \lim_{\eps\rightarrow 0}\rho_{\eps, \Dt, h}^n =\rho_{\Dt, h}^n, \quad q_{\eps, \Dt, h}^n\rightharpoonup q^n_{\Dt, h}\;\textrm{in}\; L^2(\Omega_x)\;\;\textrm{as} \;\; \eps\rightarrow 0,\;\; \forall n\geq 1
\eeq
for some $\rho_{\Dt, h}^n, q_{\Dt, h}^n \in U_h^k$, and the limits satisfy
\begin{subequations}
\label{eq:FDG:limit}
\begin{align}
\left(\frac{\rho_{\Dt, h}^{n+1}-\rho_{\Dt, h}^n}{\Dt}, \phi\right) &
-\sum_i \int_{I_i} q_{\Dt, h}^n \df_x\testR dx - \sum_i \widehat{(q_{\Dt, h}^n)}_{\iL} [\testR]_\iL=0, \quad \forall \phi\in U_h^k, \label{eq:FDG:limit:r}\\
(q_{\Dt, h}^{n+1}, \psi) &
=\langle v^2\rangle\left(\sum_i\int_{I_i} \rho_{\Dt, h}^{n+1} \df_x\testG dx + \sum_i \widehat{(\rho^{n+1}_{\Dt, h})}_\iL[\testG]_\iL    \right) ,\quad \forall \psi\in U_h^k,
 \label{eq:FDG:limit:g}
\end{align}
\end{subequations}
for $n\geq 0$, with the initial data given by \eqref{eq:init:limit}.
\end{thm}

The limiting scheme \eqref{eq:FDG:limit}, though being implicitly defined, is intrinsically explicit if one first solves $\rho_{\Dt, h}^{n+1}$ then $g_{\Dt, h}^{n+1}$ in actual implementation.  Note that this limiting scheme with any fixed $k$ is a consistent scheme for the limiting heat equation $\partial_t \rho=\partial_x(\langle v^2\rangle \partial_x\rho)$ in its first order form
\begin{equation}
\partial_t \rho+\partial_x q=0,\quad q=-\langle v^2\rangle \partial_x \rho.
\end{equation}
In fact, the spatial discretization in the limiting scheme is exactly the local DG spatial discretization for the heat equation studied in \cite{cockburn1998local}.

\begin{proof}
We start with revisiting $\gamma$ in \eqref{eq:thm1:5} and \eqref{eq:gamma} from the stability analysis. By requiring $\gamma \leq - c_0$, we obtain
the condition \eqref{eq:Dt0} on the time step $\Dt$. It is easy to see $\Dt_{stab,c_0} < \Dt_{stab}$ (In fact, $\lim_{c_0\rightarrow 0}\Dt_{stab,c_0}=\Dt_{stab}$.) The remaining of the proof consists of two steps.

 \bigskip
{\sf Step 1:}
In this step, we want to show
\beq
\label{eq:aa:0}
\sup_\eps ||\rho_{\eps, \Dt, h}^n||<\infty,\qquad \sup_\eps |||g_{\eps, \Dt, h}^n|||<\infty, \quad \forall n\geq 1,
\eeq
when $\Dt$ and $h$ satisfy \eqref{eq:Dt0}. With such mesh parameters,  equation \eqref{eq:thm1:5} turns to
\beq
\label{eq:aa:1}
\frac{1}{2\Dt}\left((||\rho_{\eps, \Dt, h}^{n+1}||^2+\eps^2|||g_{\eps, \Dt, h}^n|||^2)-(||\rho_{\eps, \Dt, h}^{n}||^2+\eps^2|||g_{\eps, \Dt, h}^{n-1}|||^2)\right)\leq -c_0 |||g_{\eps, \Dt, h}^n|||^2,
\eeq
that is
\beq
\label{eq:aa:2}
||\rho_{\eps, \Dt, h}^{n+1}||^2+(2\Dt c_0+\eps^2)|||g_{\eps, \Dt, h}^n|||^2\leq ||\rho_{\eps, \Dt, h}^{n}||^2+\eps^2|||g_{\eps, \Dt, h}^{n-1}|||^2.
\eeq
On the other hand, $\Dt_{stab,c_0} < \Dt_{stab}$, and this implies  the stability estimate \eqref{eq:thm1:0}.
Combining \eqref{eq:aa:2}, \eqref{eq:thm1:0}, the boundedness of $\sup_\eps||\rho_{\eps, \Dt, h}^{1}||, \sup_\eps|||g_{\eps, \Dt, h}^{0}|||$ in Lemma \ref{lem:init}, we will obtain \eqref{eq:aa:0}.

 \bigskip
{\sf Step 2:} Now we would like to establish the asymptotic behavior in \eqref{eq:aa:5}, as well as the fact that the limits $\rho_{\Dt, h}^n$ and $q_{\Dt, h}^n$ satisfy \eqref{eq:FDG:limit} with the initial data \eqref{eq:init:limit}.

First of all, it is easy to see that to obtain \eqref{eq:aa:5}, it is equivalent to show
\beq
\label{eq:aa:6}
 \lim_{m\rightarrow \infty}\rho_{\eps_m, \Dt, h}^n  =\rho_{\Dt, h}^n, \quad q_{\eps_m, \Dt, h}^n\rightharpoonup q^n_{\Dt, h}\;\textrm{in}\; L^2(\Omega_x)\;\;\textrm{as} \;\; m\rightarrow \infty,\;\;\forall n\geq 1
\eeq
where $\{\eps_m\}_{m=1}^{\infty}$ is any sequence such that $\lim_{m\rightarrow \infty} \eps_m=0$. Given that $U_h^k$ is finite dimensional, the boundedness of $\sup_m ||\rho_{\eps_m, \Dt, h}^n||$ from \eqref{eq:aa:0} implies that there is a subsequence $\{\rho_{\eps_{m_r}, \Dt, h}^n\}_{r=1}^{\infty}$ converging in $U_h^k$ under {\em any} norm as $r\rightarrow \infty$. Let's denote the limit as $\rho_{\Dt, h}^n\in U_h^k$.

Now we turn to $\{q_{\eps_m, \Dt, h}^n\}_{m=1}^\infty$. For the simplicity of notations, this sequence will be denoted as $\{q^n_{\eps_m}\}_{m=1}^\infty=\{\langle v g^n_{\eps_m}\rangle\}_{m=1}^\infty$ in the present paragraph. For each function $g_{\eps_m}$, it can be written as $g^n_{\eps_m}(x,v)=\sum_{j=1}^{N_k}\alpha_{\eps_m}^{(j)}(v)\Psi_j(x)$.
In addition, we have
$||| g^n_{\eps_m}|||=\left(\sum_{j=1}^{N_k} ||\alpha_{\eps_m}^{(j)}||^2_{L^2(\Omega_v)}\right)^{1/2}$. This, in addition to the boundedness of
$\sup_m |||g_{\eps_m}^n|||$ in \eqref{eq:aa:0}, indicates that $\sup_m ||\alpha_{\eps_m}^{(j)}||^2_{L^2(\Omega_v)}$, therefore $\sup_r ||\alpha_{\eps_{m_r}}^{(j)}||^2_{L^2(\Omega_v)}$  is bounded for any $j=1,\cdots, N_k$. As a Hilbert space, $L^2(\Omega_v)$ is weakly sequentially compact, that is, $\{\alpha_{\eps_{m_r}}^{(j)}\}_{r=1}^\infty$ has a subsequence which is weakly convergent in $L^2(\Omega_v)$. Without loss of generality, this subsequence is still denoted as $\{\alpha_{\eps_{m_r}}^{(j)}\}_{r=1}^\infty$, and the weak limit when $r\rightarrow \infty$ is denoted as $\alpha_0^{(j)}\in L^2(\Omega_v)$, $\forall j$. We now define $g_{\Dt, h}^n(x,v)=\sum_{j=1}^{N_k}\alpha_0^{(j)}(v)\Psi_j(x)$, and $q_{\Dt, h}^n= \langle v g_{\Dt, h}^n\rangle= \sum_{j=1}^{N_k}\langle v \alpha_0^{(j)}\rangle\Psi_j(x)$. For any $\psi\in U_h^k$,
\begin{equation*}
\lim_{r\rightarrow \infty}(q_{\eps_{m_r}, \Dt, h}^n, \psi)=\sum_{j=1}^{N_k} \left(\lim_{r\rightarrow\infty}\langle v \alpha_{\eps_{m_r}}^{(j)}\rangle \right)(\Psi_j(x), \psi)=\sum_{j=1}^{N_k} \left(\langle v \alpha_0^{(j)}\rangle \right)(\Psi_j(x), \psi)=(q_{\Dt, h}^n, \psi).
\end{equation*}

Up to now, we have shown that \eqref{eq:aa:5} holds for a subsequence of $\{\rho_{\eps_{m_r}, \Dt, h}^n\}_{r=1}^\infty$ and $\{q_{\eps_{m_r}, \Dt, h}^n\}_{r=1}^\infty$ as $r\rightarrow \infty$. Moreover, it is straightforward to see that the limits $\rho_{\Dt, h}^n$ and $g_{\Dt, h}^n$, $n=1, 2, \cdots$, satisfy \eqref{eq:FDG:limit} with the initial data \eqref{eq:init:limit}. On the other hand, given the initial data \eqref{eq:init:limit}, the solution to \eqref{eq:init:limit} at $t^{n+1}$, $n\geq 0$, is uniquely determined by first solving \eqref{eq:FDG:limit:r} for $\rho^{n+1}_{\Dt, h}$ and then solving  \eqref{eq:FDG:limit:g} for $q^{n+1}_{\Dt, h}$.  Finally, we can follow a standard contradiction argument and the uniqueness of the solution to \eqref{eq:FDG:limit} to conclude that \eqref{eq:aa:5} holds for the entire sequence.
\end{proof}

\begin{rem}
The rigorous asymptotic analysis is established for the methods with the first order accuracy in time  \eqref{eq:FDG:1T}.
 When higher order temporal discretizations are used as in \cite{JLQX_numerical}, one can follow the steps in this subsection to obtain a rigorous asymptotic analysis, as long as a stability estimate similar to \eqref{eq:thm1:5} is available.
\end{rem}

%% file: acknowledge.tex
\bigskip
\noindent
{\bf Acknowledgement.}
This project was initiated during the authors' participation at the ICERM Semester Program on ``Kinetic Theory and Computation" in the fall of 2011. The authors want to thank for the generous support from the Institute. Part of the work was done at MFO in Oberwolfach during a Research in Pairs program. The first three authors appreciate the support and hospitality of the Institute.

%The authors thank the support from the institute of computational and experimental research in mathematics (ICERM) in Brown University and the support from Oberwolfach.....

%% file: paper.bbl
\begin{thebibliography}{10}

\bibitem{bardos1991fluid}
{\sc C.~Bardos, F.~Golse, and D.~Levermore}, {\em Fluid dynamic limits of
  kinetic equations. {I}. formal derivations}, Journal of Statistical Physics,
  63 (1991), pp.~323--344.

\bibitem{bardos1984diffusion}
{\sc C.~Bardos, R.~Santos, and R.~Sentis}, {\em Diffusion approximation and
  computation of the critical size}, Transactions of the American Mathematical
  Society, 284 (1984), pp.~617--649.

\bibitem{boscarino2011implicit}
{\sc S.~Boscarino, L.~Pareschi, and G.~Russo}, {\em Implicit-explicit
  {R}unge--{K}utta schemes for hyperbolic systems and kinetic equations in the
  diffusion limit}, SIAM Journal on Scientific Computing, 35 (2013),
  pp.~A22--A51.

\bibitem{boscarino2013flux}
{\sc S.~Boscarino and G.~Russo}, {\em Flux-explicit {IMEX} {R}unge--{K}utta
  schemes for hyperbolic to parabolic relaxation problems}, SIAM Journal on
  Numerical Analysis, 51 (2013), pp.~163--190.

\bibitem{carrillo2008numerical}
{\sc J.~A. Carrillo, T.~Goudon, P.~Lafitte, and F.~Vecil}, {\em Numerical
  schemes of diffusion asymptotics and moment closures for kinetic equations},
  Journal of Scientific Computing, 36 (2008), pp.~113--149.

\bibitem{cercignani1988boltzmann}
{\sc C.~Cercignani}, {\em The Boltzmann equation}, Springer, 1988.

\bibitem{ciarlet:book}
{\sc P.~G. Ciarlet}, {\em The finite element method for elliptic problems},
  North Holland, 1975.

\bibitem{cockburn1998local}
{\sc B.~Cockburn and C.-W. Shu}, {\em The local discontinuous {G}alerkin method
  for time-dependent convection-diffusion systems}, SIAM Journal on Numerical
  Analysis, 35 (1998), pp.~2440--2463.

\bibitem{guermond2010asymptotic}
{\sc J.-L. Guermond and G.~Kanschat}, {\em Asymptotic analysis of upwind
  discontinuous {G}alerkin approximation of the radiative transport equation in
  the diffusive limit}, SIAM Journal on Numerical Analysis, 48 (2010),
  pp.~53--78.

\bibitem{JLQX_numerical}
{\sc J.~Jang, F.~Li, J.-M. Qiu, and T.~Xiong}, {\em High order asymptotic
  preserving {DG-IMEX} schemes for discrete-velocity kinetic equations in a
  diffusive scaling}, http://arxiv.org/abs/1306.0227, submitted,  (2013).

\bibitem{jin2010asymptotic}
{\sc S.~Jin}, {\em Asymptotic preserving ({AP}) schemes for multiscale kinetic
  and hyperbolic equations: a review}, Lecture Notes for Summer School on
  Methods and Models of Kinetic Theory (M\&MKT), Porto Ercole (Grosseto,
  Italy),  (2010).

\bibitem{jin1996numerical}
{\sc S.~Jin and C.~Levermore}, {\em Numerical schemes for hyperbolic
  conservation laws with stiff relaxation terms}, Journal of Computational
  Physics, 126 (1996), pp.~449--467.

\bibitem{jin1998diffusive}
{\sc S.~Jin, L.~Pareschi, and G.~Toscani}, {\em Diffusive relaxation schemes
  for multiscale discrete-velocity kinetic equations}, SIAM Journal on
  Numerical Analysis, 35 (1998), pp.~2405--2439.

\bibitem{jin2000uniformly}
\leavevmode\vrule height 2pt depth -1.6pt width 23pt, {\em Uniformly accurate
  diffusive relaxation schemes for multiscale transport equations}, SIAM
  Journal on Numerical Analysis, 38 (2000), pp.~913--936.

\bibitem{klar1998asymptotic}
{\sc A.~Klar}, {\em An asymptotic-induced scheme for nonstationary transport
  equations in the diffusive limit}, SIAM Journal on Numerical Analysis, 35
  (1998), pp.~1073--1094.

\bibitem{lafitte2012asymptotic}
{\sc P.~Lafitte and G.~Samaey}, {\em Asymptotic-preserving projective
  integration schemes for kinetic equations in the diffusion limit}, SIAM
  Journal on Scientific Computing, 34 (2012), pp.~579--602.

\bibitem{morel1989asymptotic}
{\sc E.~W. Larsen and J.~Morel}, {\em Asymptotic solutions of numerical
  transport problems in optically thick, diffusive regimes {II}}, Journal of
  Computational Physics, 83 (1989), pp.~212--236.

\bibitem{larsen1987asymptotic}
{\sc E.~W. Larsen, J.~Morel, and W.~F. Miller~Jr}, {\em Asymptotic solutions of
  numerical transport problems in optically thick, diffusive regimes}, Journal
  of Computational Physics, 69 (1987), pp.~283--324.

\bibitem{lemou2010new}
{\sc M.~Lemou and L.~Mieussens}, {\em A new asymptotic preserving scheme based
  on micro-macro formulation for linear kinetic equations in the diffusion
  limit}, SIAM Journal on Scientific Computing, 31 (2010), pp.~334--368.

\bibitem{liu2010analysis}
{\sc J.~Liu and L.~Mieussens}, {\em Analysis of an asymptotic preserving scheme
  for linear kinetic equations in the diffusion limit}, SIAM Journal on
  Numerical Analysis, 48 (2010), pp.~1474--1491.

\bibitem{liu2004boltzmann}
{\sc T.-P. Liu and S.-H. Yu}, {\em Boltzmann equation: micro-macro
  decompositions and positivity of shock profiles}, Communications in
  Mathematical Physics, 246 (2004), pp.~133--179.

\bibitem{lowrie2002methods}
{\sc R.~Lowrie and J.~Morel}, {\em Methods for hyperbolic systems with stiff
  relaxation}, International Journal for Numerical Methods in Fluids, 40
  (2002), pp.~413--423.

\bibitem{naldi1998numerical}
{\sc G.~Naldi and L.~Pareschi}, {\em Numerical schemes for kinetic equations in
  diffusive regimes}, Applied Mathematics Letters, 11 (1998), pp.~29--35.

\bibitem{pareschi2011efficient}
{\sc L.~Pareschi and G.~Russo}, {\em Efficient asymptotic preserving
  deterministic methods for the {B}oltzmann equation}, AVT-194 RTO AVT/VKI,
  Models and Computational Methods for Rarefied Flows, Lecture Series held at
  the von Karman Institute, Rhode St. Gense, Belgium,  (2011).

\bibitem{reyna2013bound}
{\sc M.~Reyna and F.~Li}, {\em {Operator bounds and time step conditions for DG
  and central DG methods}}, Journal of Scientific Computing, in print,  (2014).

\bibitem{saint2009hydrodynamic}
{\sc L.~Saint-Raymond}, {\em Hydrodynamic limits of the Boltzmann equation},
  vol.~1971, Springer, 2009.

\bibitem{warburton2008taming}
{\sc T.~Warburton and T.~Hagstrom}, {\em Taming the {CFL} number for
  discontinuous {G}alerkin methods on structured meshes}, SIAM Journal on
  Numerical Analysis, 46 (2008), pp.~3151--3180.

\bibitem{zhang2004error}
{\sc Q.~Zhang and C.-W. Shu}, {\em Error estimates to smooth solutions of
  {R}unge--{K}utta discontinuous {G}alerkin methods for scalar conservation
  laws}, SIAM Journal on Numerical Analysis, 42 (2004), pp.~641--666.

\bibitem{zhang2010stability}
\leavevmode\vrule height 2pt depth -1.6pt width 23pt, {\em Stability analysis
  and a priori error estimates of the third order explicit {R}unge-{K}utta
  discontinuous {G}alerkin method for scalar conservation laws}, SIAM Journal
  on Numerical Analysis, 48 (2010), pp.~1038--1063.

\end{thebibliography}
